\documentclass[12pt]{article}
\usepackage{graphicx,amsthm,amsmath,amssymb,enumerate,color}

\RequirePackage[colorlinks]{hyperref}

\newtheorem{thm}{Theorem}[section]
\newtheorem{lem}[thm]{Lemma}
\newtheorem{cor}[thm]{Corollary}
\newtheorem{prop}[thm]{Proposition}

\theoremstyle{remark}

\newcommand{\Cov}{\mathrm {cov}}
\newcommand{\cov}{{\mathrm {cov}}}
\newcommand{\Var}{\mathrm {var}}
\newcommand{\var}{\mathrm {var}}
\newcommand{\Exp}{{\mathbb E}}
\newcommand{\esp}{{\mathbb E}}
\newcommand{\pr}{{\mathbb P}}

\def\d{{\mathrm d}}

\newcommand{\convdistr}{\stackrel{\mbox{\tiny\rm d}}{\to}}
\newcommand{\convprob}{\stackrel{\mbox{\small\tiny p}}{\to}}

\begin{document}

\title{The tail empirical process for long memory stochastic volatility
  sequences\protect}

\author{Rafa{\l} Kulik\thanks{
Corresponding author: Department of Mathematics and Statistics,
University of Ottawa, 585 King Edward Avenue, Ottawa ON K1N 6N5,
Canada, email: rkulik@uottawa.ca, phone 1 613 562 5800 Ext. 3526.}
\and Philippe Soulier\thanks{ Department of Mathematics, University
Paris X, Batiment G, Bureau E18 200, Avenue de la R?publique 92000,
Nanterre Cedex, France, email: philippe.soulier@u-paris10.fr} }

\maketitle

\begin{abstract}
  This paper describes the limiting behaviour of tail empirical processes
  associated with long memory stochastic volatility models. We show that such a
  process has dichotomous behaviour, according to an interplay between the
  Hurst parameter and the tail index.
  %In particular, the limit may be non-Gaussian and/or
  %degenerate, indicating an influence of long memory.
  On the other hand, the tail empirical process with random levels never
  suffers from long memory. This is very desirable from a practical point of
  view, since such a process may be used to construct the Hill estimator of the
  tail index. To prove our results we need to establish new results for
  regularly varying distributions, which may be of independent
  interest.
\end{abstract}

\section{Introduction}
The goal of this article is to study weak convergence results for
the tail empirical process associated with some long memory
sequences.  Besides of theoretical interests on its own, the results
are applicable in different statistical procedures based on several
extremes. A similar problem was studied in case of independent,
identically distributed random variables in \cite{Einmahl1992}, or
for weakly dependent sequences in \cite{Drees2000},
\cite{Drees2002}, \cite{Drees2003}, \cite{Rootzen2009}.
% The reader is referred also to
% an excellent presentation in \cite{Resnick2007}.
% ph: only in the iid case!!

Our set-up is as follows.  Assume that $\{X_i \,, \ i \in \mathbb
Z\}$, is a stationary Gaussian process with unit variance and
covariance
\begin{gather}
  \label{eq:lrd}
  \rho_{i-j} = \Cov(X_i,X_j) = |i-j|^{2H-2} \ell_0(|i-j|) \; ,
\end{gather}
where $H\in (1/2,1)$ is the Hurst exponent and $\ell_0$ is a slowly
varying function at infinity, i.e. $\lim_{t\to\infty}
\ell_0(tx)/\ell_0(x)=1$ for all $x>0$. The sequence in this case is
referred to as an \textit{LRD} Gaussian sequence. We also consider
weakly dependent Gaussian sequences, i.e. such that
$\sum_{j=1}^{\infty}|\Cov(X_1,X_{j+1})|<\infty$.

We shall consider a stochastic volatility process defined as
$$
Y_i=\sigma(X_i)Z_i,\qquad i \in \mathbb Z,
$$
where $\sigma(\cdot)$ is a nonnegative, deterministic function and
that $\{Z,Z_i$, $i \in \mathbb Z\}$, is a sequence of i.i.d. random
variables, independent of the process $\{X_i\}$. We note, in
particular, that if $\Exp[Z^2]<\infty$ and $\Exp[Z]=0$, then the
$Y_i$s are uncorrelated, no matter  the assumptions on
dependence structure of the underlying Gaussian sequence.\

Stochastic volatility models have become popular in financial time
series modeling. In particular, if $H\in (1/2,1)$, these models are
believed to capture two standardized features of financial data: long
memory of squares or absolute values, and conditional
heteroscedascity. If $\sigma(x)=\exp(x)$, then the model is referred
to in the econometrics literature as \textit{Long Memory in
  Stochastic Volatility} (LMSV) and was introduced in
\cite{BreidtCratoLima1998}. For an overview of stochastic volatility
models with long memory we refer to
\cite{DeoHsiehHurvichSoulier2006}.

Let $F=F_i$, $i\ge 1$, be the marginal distribution of $Y_i$. We
want to consider the case where $F$ belongs to the domain of
attraction of an extreme value distribution with positive index
$\gamma$, i.e. there exist sequences $u_n$, $n\ge 1$,
$u_n\to\infty$, and $\sigma_n$, $n\ge 1$, such that the associated
conditional tail distribution function
\begin{equation}
  \label{eq:Tn}
  T_n(x)  = \frac{\bar F(u_n+\sigma_n x)}{\bar F(u_n)},
  \qquad  x\ge 0, \; n \ge 1,
\end{equation}
satisfies
\begin{equation}\label{eq:cond-tail}
  \lim_{n\to\infty}   T_n(x) = T(x) = \left(1+x\right)^{-1/\gamma}\; , \quad x\geq0
\;  .
\end{equation}
For the stochastic volatility model, this will be obtained through a
further specification.  Let $F_Z$ be the marginal distribution of
the noise sequence. We will assume that for some $\alpha\in
(0,\infty)$,
\begin{equation}\label{eq:Pareto-assumption}
  \bar F_Z(z) = \pr(Z > x) = x^{-\alpha} \ell(x) \; ,
\end{equation}
where $\ell$ is again a slowly varying function.  Assuming
(\ref{eq:Pareto-assumption}) and
$\Exp[\sigma^{\alpha+\epsilon}(X_1)]<\infty$ for some $\epsilon>0$,
we conclude by Breiman's Lemma \cite{Breiman1965} (see also
\cite[Proposition 7.5]{resnick:2007}) that
$$
\bar F(x) = \pr(Y_1>x) = \pr(\sigma(X_1)Z_1>x) \sim
\Exp[\sigma^{\alpha}(X_1)] \pr(Z_1>x) \; , \ \mbox{ as } x\to\infty.
$$
Consequently, $\bar F(\cdot)$ satisfies (\ref{eq:cond-tail}) with
$\sigma_n=u_n$ and $\gamma=1/\alpha$.

Similarly to \cite{Rootzen2009}, we define the tail empirical
distribution function and the tail empirical process, respectively,
as
$$
\tilde T_n(s) = \frac{1}{n\bar F(u_n)} \sum_{j=1}^n %
1_{\{Y_j
  > u_n+ u_n s\}} \; ,
$$
and
\begin{equation}\label{eq:te-def}
  e_n(s) = \tilde T_n(s)-T_n(s) \; , \ s \in [0,\infty) \; .
\end{equation}
From \cite{Rootzen2009} we conclude that under appropriate mixing
and other conditions on a stationary sequence $Y_i$, $i\ge 1$, the
tail empirical process converges weakly and the limiting covariance
is affected by dependence. In our case, the results
\cite{Rootzen2009} do not seem applicable. In fact, it will be shown
that we have two different modes of convergence. If $u_n$ is
\textit{large}, then $\sqrt{n\bar F(u_n)}$ is the proper scaling
factor and the limiting process is Gaussian with the same covariance
structure as in case of i.i.d.~random variables $Y_i$. Otherwise, if
$u_n$ is \textit{small}, then the limit is affected by long memory
of the Gaussian sequence. The scaling is different and the limit may
be non-normal. These results are presented in Section
\ref{sec:general}. Note that a similar dichotomous phenomenon was
observed in the context of sums of extreme values associated with
long memory moving averages, see \cite{Kulik2008a} for more details.
On the other hand, this dichotomous behaviour is in contrast with
the convergence of point processes based on stochastic volatility
models with regularly varying innovations, where (long range)
dependence does not affect the limit (See \cite{DavisMikosch2001}).

The process $e_n(\cdot)$ is unobservable in practice, since the
parameter $u_n$ depends on the unknown distribution $F$. Also, $u_n$
being \textit{large} or \textit{small} depends on a delicate balance
between the tail index $\alpha$ and the Hurst parameter $H$. In
order to overcome this, we consider as in \cite{Rootzen2009} a
process with random levels. There, we set $k=n\bar F(u_n)$ and
replace the deterministic level $u_n$ by $Y_{n-k:n}$, where
$Y_{n:n}\ge Y_{n-1:n}\ge \cdots\ge Y_{1:n}$ are the increasing order
statistics of the sample $Y_1,\dots,Y_n$.  The number $k$ can be
thought as the number of extremes used in a construction of the tail
empirical process. It turns out that if the number of extremes is
{\it small} (which corresponds to a {\it
  large} $u_n$ above), then the limiting process changes as compared to the one
associated with $e_n(\cdot)$, but the speed of convergence remains
the same. This has been already noticed in \cite{Rootzen2009} in the 
weakly dependent case. On the other hand, if $k$ is {\it large},
then the scaling from $e_n(\cdot)$ is no longer correct (see
Corollary \ref{cor:practical}). In fact, the process with random
levels has a faster rate of convergence and we claim in Theorem
\ref{thm:practical-1} that the rate of convergence and the limiting
process are not affected at all by long memory, provided that a
technical second order regular variation condition is fulfilled. The
reader is referred to Section \ref{sec:random-levels}. On the other
hand, it should be pointed out that our results are for \textit{the}
long memory stochastic volatility models. It is not clear for us
whether such phenomena will be valid for example for subordinated
long memory Gaussian sequences with infinite variance.

The results for the tail empirical process $e_n(\cdot)$ allow us to
obtain asymptotic normality and non-normality of intermediate
quantiles, as described in Corollary \ref{coro:intermediate}. On the
other hand, the tail empirical process with random levels allows the
study of the Hill estimator of the tail index $\alpha$ (Section
\ref{sec:Hill}). Consequently, as  shown in Corollary
\ref{cor:Hill}, long memory does not have influence on its
asymptotic behaviour. These theoretical observations are justified
by simulations in Section \ref{sec:numstudies}.

Last but not least, we have some contribution to the theory of
regular variation. To establish our results in the random level
case, we need to work under a second order regular variation
condition. Consequently, one has to establish in a Breiman's-type
lemma that such a condition is transferable from $\bar F_Z$ to $\bar
F$. This is done in Section \ref{sec:s-o-c}.

\section{Results}
\subsection{Tail empirical process}
\label{sec:general} Let us define a function $G_n$ on
$(-\infty,\infty)\times[0,\infty)$ by
\begin{equation}\label{eq:function-gn}
  G_n(x,s) = \frac{\pr(\sigma(x)Z_1>(1+s)u_n)}{\pr(Z_1>u_n)} \; .
\end{equation}
By Breiman's Lemma and the regular variation of $\bar F_Z$, we
conclude that for each $s\in [0,1]$, this function converges
pointwise to $T(s)G(x)$, where $G(x)=\sigma^{\alpha}(x)$.  A
stronger convergence can actually be proved (see Section
\ref{sec:proof-s-o-c} for a proof).
\begin{lem} \label{lem:convergence-uniforme-hermite}
  If~(\ref{eq:Pareto-assumption}) holds and
  $\esp[\sigma^{\alpha+\epsilon}(X)]<\infty$ for some $\epsilon>0$, then
  \begin{align}
    \lim_{n\to\infty} \esp \left[ \sup_{s\geq0} \left|
        G_n(X,s)-\sigma^\alpha(X) T(s) \right|^p \right] = 0 \;
  \end{align}
  for all $p$ such that $p\alpha<\alpha+\epsilon$.
\end{lem}
In order to introduce our assumptions, we need to define the Hermite
rank of a function.  Recall that the Hermite polynomials $H_m$,
$m\geq0$, form an orthonormal basis of the set of functions $h$ such
that $\esp[h^2(X)]<\infty$, where $X$ denotes a generic standard
Gaussian random variable (independent of all other random variables
considered here), and have the following properties:
\begin{align*}
  \esp[H_m(X)] = 0 \; , \ m\geq 1 \; , \ \cov(H_j(X),H_k(X)) =
  \delta_{j,k} k! \;
\end{align*}
where $\delta_{j,k}$ is Kronecker's delta, equal to 1 if $j=k$ and
zero otherwise. Then $h$ can be expanded as
\begin{align*}
  h = \sum_{m=0}^\infty \frac{c_m}{m!} H_m \; ,
\end{align*}
with $c_m = \esp[h(X) H_m(X)]$ and the series is convergent in the
mean square. The smallest index $m\geq1$ such that $c_m\ne0$ is
called the Hermite rank of $h$.  Note that with this definition, the
Hermite rank is always at least equal to one and the Hermite rank of
a function $h$ is the same as that of $h-\esp[h(X)]$.
%where $X$ is $N(0,1)$.

Let $J_n(m,s)$ denote the Hermite coefficients of the function $x\to
G_n(x,s)$. Since $\Exp[|H_m(X_1)|^r]<\infty$ for all $r\geq1$,
Lemma~\ref{lem:convergence-uniforme-hermite} implies that the
Hermite coefficients $J_n(m,s)$ converge to $J(m) T(s)$, where
$J(m)$ is the $m$-th Hermite coefficient of $G$, uniformly with
respect to $s\geq0$. This implies that for large $n$, the Hermite
rank of $G_n(\cdot,s)$ is not bigger than the Hermite rank of $G$.
In order to simplify the proof of our results, we will use the
following assumption, which is not very restrictive.

\paragraph{Assumption (H)} Denote by $J_{n}(m,s)$, $m\ge 1$,
the Hermite coefficients of $G_n(\cdot,s)$ and let $q_n(s)$ be the
Hermite rank of $G_n(\cdot,s)$. Define
$$
q_n=\inf_{s\geq0} q_n(s)  \; ,
$$
the Hermite rank of the class of functions
$\{G_n(\cdot,s),s\geq0\}$. In other words, the number $q_n$ is the
smallest $m$ such that $J_n(m,s)\not=0$ for at least one $s$.
Furthermore, let $q$ be the Hermite rank of $G$. We assume that
$q_n=q$ for $n$ large enough.

\noindent{\em Remark}. \ Since for large enough $n$ it holds that
$q_n(s) \leq q$ for all $s$, the assumption is fulfilled, for
example, when $G$ has Hermite rank 1 (as is the case for the
function $x\to\mathrm e^x$), or if the function $\sigma$ is even
with the Hermite rank~2.

%In order to prove tightness, we will also need the following
%condition.
%\begin{gather}
%  \label{eq:second-order-tightness}
%  \exists C>0 \;, \ \ \forall y \geq 1 \; , \ \ \ \forall t \geq s > 0 \;, \
%\   \frac{\pr(sy < Y \leq ty)}{\pr(Y>y)} \leq C(t-s) \; .
%\end{gather}
%The condition~(\ref{eq:second-order-tightness}) is unprimitive since it is
%expressed in terms of $Y$. It holds if
%%$\esp[\sigma^{\alpha+\epsilon}(X)]<\infty$ for some $\epsilon>0$ and $Z$
%satisfies the following condition. There exists a constant $C$ such that for
%large enough $y$,
%  \begin{align}
%  \label{eq:second-order-tightness-Z}
%  \exists C>0 \;, \ \ \exists y_0>0 \; , \ \ \forall y \geq y_0 \; , \ \ \
%  \forall t \geq s > 0 \;, \ \ \frac{\pr(sy < Z \leq ty)}{\pr(Z>y)} \leq C
%  (s\wedge1)^{-\alpha-1-\epsilon} (t-s) \; .
%\end{align}
%Condition (\ref{eq:second-order-tightness-Z}) holds in particular if
%$Z$ has an ultimately monotone density, which is then necessarily
%regularly varying at infinity with index $-\alpha-1$ by the monotone
%density Theorem, see
%\cite[Theorem~1.7.2]{bingham:goldie:teugels:1989}. See
%Section~\ref{sec:s-o-c} for a second order regular condition that
%implies~(\ref{eq:second-order-tightness-Z}).
%\end{rem}

The result for the general tail empirical process is as follows.
\begin{thm}
  \label{thm:general}
  Assume {\rm (H)} with $q(1-H) \ne 1/2$,~(\ref{eq:lrd}),
  (\ref{eq:Pareto-assumption}), $n\bar F(u_n)\to\infty$ and that there exists
  $\epsilon>0$ such that
  \begin{align} \label{eq:moment-2alpha+epsilon}
    0 < \Exp[\sigma^{2\alpha+\epsilon}(X_1)]<\infty \; .
  \end{align}
\begin{enumerate}[(i)]
\item \label{item:big-u} If $n \bar F(u_n) \rho_n^q \to 0$ as $n\to\infty$ or
  if $\{X_j\}$ is weakly dependent, then $\sqrt{n\bar F(u_n)} \, e_n$ converges
  weakly in $D([0,\infty))$ to the Gaussian
  process $W \circ T$, where $W$ is the standard Brownian motion.
  % $e$ with covariance $c(x,y)=\textcolor{black}{T(x\vee y)}$.
 % If moreover (\ref{eq:second-order-tightness}) holds, then the
 % convergence holds in $D([0,\infty))$.
\item \label{item:small-u} If $n \bar F(u_n) \rho_n^q \to \infty$ as
  $n\to\infty$ then $\rho_n^{-q/2} e_n$ converges weakly in $D([0,\infty))$ to the
process
  $(\Exp[\sigma^{\alpha}(X_1)])^{-1} J(q) T L_q$, where the random
  variable $L_q$ is defined in~(\ref{eq:def-L_q}).
  %If moreover
  %(\ref{eq:second-order-tightness}) holds, then the convergence holds
  %in $D([0,\infty))$.
\end{enumerate}
\end{thm}

\subsubsection*{Remarks}
\begin{enumerate}[-]
\item We rule out the borderline case $q(1-H)=1/2$ for the sake of
  brevity and simplicity of exposition.  It can be easily shown that
  if $q(1-H)=1/2$, then $\sqrt {n\bar F(u_n)} e_n$ converges to $W
  \circ T$ provided $1/\bar F(u_n)$ tends to infinity faster than a
  certain slowly varying function (e.g.  if $u_n=n^\gamma$ for some
  $\gamma>0$), even though it may hold in this case that
  $n\rho_n^q\to\infty$. The reason is that the variance of the partial
  sums of $G(X_k)$ is of order $n$ times a slowly varying function
  which dominates $\ell_0^q(n)$.
\item
% \begin{rem}
  Here $D([0,\infty)$ is endowed with Skorohod's $J_1$ topology, and tightness
  is checked by applying \cite[Theorem 15.6]{billingsley:1968}.  Since the
  limiting processes have almost surely continuous paths, this convergence
  implies uniform convergence on compact sets of $[0,\infty)$. See also
  \cite{whitt:2002}.
% \end{rem}
 \item
% \begin{rem}
  The meaning of the above result is that for $u_n$ \textit{large}, long
  memory does not play any role. However, if $u_n$ is \textit{small},
  long memory comes into play and the limit is degenerate.
  Furthermore, in the case of Theorem \ref{thm:general},
  \textit{small} and \textit{large} depend on the relative behaviour of
  the tail of $Y_1$ and the memory parameter. Note that the condition
  $n\bar F(u_n) \rho_n^q\to\infty$ implies that $1-2q(1-H)>0$, in
  which case the partial sums of the subordinate process $\{G(X_i)\}$
  weakly converge to the Hermite process of order $q$ (see
  Section~\ref{sec:LRD-Gaussian}).  The cases (\ref{item:big-u}) and
  (\ref{item:small-u}) will be referred to as the limits \textit{in
    the i.i.d. zone} and \textit{in the LRD zone}, respectively.
%\end{rem}
 \item
%\begin{rem}
   Condition $\Exp[\sigma^{\alpha+\epsilon}(X_1)]<\infty$ is standard when one
   deals with regularly varying tails. However, we need the condition
   $\Exp[\sigma^{2\alpha+\epsilon}(X_1)]<\infty$ in order to obtain the
   limiting distributions in the i.i.d. and LRD zones. See
   section~\ref{sec:fidi}.
% \end{rem}
 \item
% \begin{rem}
  The result should be extendable to general, not necessary Gaussian,
  long memory linear sequences. Instead of the limit theorems and
  covariance bounds of Section \ref{sec:LRD-Gaussian}, one can use
  limit theorems from \cite{HoHsing1997}, and the covariance bounds of
  \cite[Lemma 3]{GiraitisSurgailis1999}.
% \end{rem}
\item
%\begin{rem} \label{rem:Rootzen-condition}
  Rootzen \cite{Rootzen2009} obtained asymptotic the behaviour of the tail
  empirical process of a general stationary sequence $\{\mathcal Y_j\}$ under, in
  particular, the following conditions (see \cite[Section 4]{Rootzen2009}):
\begin{itemize}
\item $l_n=o(r_n)$, $r_n=o(n)$;
\item[{\rm (C1)}] $\Exp[|N_n(x,y)|^p|N_n(x,y)\not=0]\le \infty$, where $p> 2$
  and $N_n$ is the point process of exceedances;
\item[{\rm (C2)}] $\beta_n(l_n)n/r_n\to 0$, where $\beta_n(\cdot)$ is the
$\beta$-mixing coefficient w.r.t. sigma field generated by the random
variables $\mathcal Y_j1_{\{\mathcal Y_j>u_n\}}$;
\item[{\rm (C3)}]
$$
\frac{1} {r_n\bar F(u_n)} \Cov\left(\sum_{i=1}^{r_n} 1_{\{\mathcal Y_i>u_n(1+s)\}},
  \sum_{j=1}^{r_n} 1_{\{\mathcal Y_j>u_n(1+t)\}}\right) \to r(x,y),
$$
for some function $r(x,y)$.
\end{itemize}
Assume that $r_n\to\infty$, $r_n=o(n)$. For the sequence $\{Y_j\}$ under
consideration here, it can be computed (see Section \ref{sec:heuristic})
\begin{align*}
  \frac{1}{r_n \bar F(u_n)} & \Cov \left(\sum_{i=1}^{r_n} 1_{\{Y_i>u_n(1+s)\}},
    \sum_{j=1}^{r_n} 1_{\{Y_j>u_n(1+t)\}} \right)  \\
  & \sim T(s\vee t) + \frac{T(s) T(t) J^2(q) r_n\bar F(u_n) \rho_{r_n}^q
  }{\Exp^2[\sigma^{\alpha}(X_1)]q!(1-2q(1-H))} \; . %\label{eq:Rootzen}
\end{align*}
Now, using (\ref{eq:lrd}), $r_n\bar F(u_n) \rho_{r_n}^q\sim \bar
F(u_n)r_n^{1-2q(1-H)}$. Since $r_n=o(n)$, then the second part
converges 0 under the condition $n\bar F(u_n)\rho_n^{q}\to 0$.
Consequently, Case (\ref{item:big-u}) guarantees that the condition
(C3) is fulfilled. As for the mixing property (C2), it is
usually established by proving the standard $\beta$-mixing, i.e. the
one defined in terms of random variables $Y_j$, not
$Y_j1_{\{Y_j>u_n\}}$. Now, if $\{X_j\}$ is $\beta$-mixing (in the
latter sense) with rate $\beta_n$, then the same holds for
$\{Y_j\}$. In our case, the sequence $\{X_j\}$ has long memory, and
thus it cannot be $\beta$-mixing. Therefore, it is very doubtful
that (C2) can be verified.

Note also that in the case $\sum_{j=1}^\infty
|\cov(X_0,X_j)|<\infty$, which we refer to as the short memory case,
the conclusion of part~\eqref{item:big-u} of Theorem holds without
any additional (mixing) assumption on the Gaussian process
$\{X_j\}$.

Moreover, results in the LRD zone cannot be obtain by applying
Rootzen's or any other results for weakly dependent sequences.
\end{enumerate}

\subsection{Random levels}\label{sec:random-levels}
Similarly to \cite{Rootzen2009}, we consider the case of random
levels.  Let $\Rightarrow$ denote weak convergence in
$D([0,\infty))$.  Define the increasing function $U$ on $[1,\infty)$
by $U(t) = F^\leftarrow(1-1/t)$, where $F^\leftarrow$ is the
left-continuous inverse of $F$.  Let $k$ denote a sequence of
integers depending on $n$, where the dependence in $n$ is omitted
from the notation as customary, and such that
\begin{align} \label{eq:k-intermediate}
  \lim_{n\to\infty} k = \lim_{n\to\infty} n/k = \infty \; .
\end{align}
Such a sequence is usually called an intermediate sequence.  Define
$u_n = U(n/k)$. If $F$ is continuous, then $n\bar F(u_n) = k$,
otherwise, since $\bar F$ is regularly varying, it holds that
$\lim_{n
  \to\infty} k^{-1} n\bar F(u_n)=1$. Thus, we will assume without loss
of generality that $k=n\bar F(u_n)$ holds. Then the statements of
Theorem \ref{thm:general} may be written respectively as
\begin{gather}
  \sqrt{k}(\tilde T_n-T_n) \Rightarrow W\circ T \; ,
  \label{eq:iid-zone-k}  \\
  \rho_n^{-q/2}(\tilde T_n - T_n) \Rightarrow
  \frac{J(q)}{\Exp[\sigma^{\alpha}(X_1)]} \; T \cdot L_q \; .
  \label{eq:LRD-zone-k}
\end{gather}
Let us rewrite the statements of (\ref{eq:iid-zone-k}),
(\ref{eq:LRD-zone-k}) as
$$
w_n(\tilde T_n-T_n) \Rightarrow w \; ,
$$
where
\begin{gather}
  w_n = \sqrt{k} \  \ \mbox{ if } \ \ \lim_{n\to\infty}
  k \rho_n^{q}  = 0 \; , \label{eq:small-k}  \\
  w_n=\rho_n^{-q/2} \ \ \mbox{ if } \ \  \lim_{n\to\infty} k \rho_n^{q} =
  \infty \; , \label{eq:big-k}
\end{gather}
and $w=W\circ T$ if~(\ref{eq:small-k}) holds (i.i.d. zone) and $w =
(\Exp[\sigma^{\alpha}(X_1)])^{-1} J(q) T L_q$ if~(\ref{eq:big-k})
holds (LRD zone).

We now want to center the tail empirical process at $T$ instead of
$T_n$. To this aim, we introduce an unprimitive second order
condition.
\begin{align}
  \lim_{n\to\infty} w_n \|T_n - T\|_\infty = 0 \; ,
  \label{eq:second-ordre-unprimitive}
\end{align}
where
\begin{align*}
  \|T_n - T\|_\infty = \sup_{t\geq1} \left| \frac{\pr ( \sigma(X) Z >
      u_n t )} {\pr ( \sigma(X) Z > u_n )} - t^{-\alpha} \right| \; .
\end{align*}
The following result is a straightforward corollary of
Theorem~\ref{thm:general}.
\begin{cor} \label{coro:centre}
  Under the assumptions of Theorem~\ref{thm:general}, if
  moreover~(\ref{eq:second-ordre-unprimitive}) holds, then $w_n(\tilde
  T_n - T)$ converges weakly in $D([0,\infty))$ to the process $w$.
\end{cor}
Let $Y_{n:1} \leq \cdots \leq Y_{n:n}$ be the increasing order
statistics of $Y_1,\dots, Y_n$.  The former result and Verwaat's
Lemma \cite[Proposition 3.3]{resnick:2007} yield the convergence of
the intermediate quantiles.
\begin{cor}
  \label{coro:intermediate}
  Under the assumptions of Corollary~\ref{coro:centre},
  $w_n(Y_{n:n-k}-u_n)/u_n$ converges weakly to $\gamma w(1)$.
\end{cor}

Define
\begin{align*} % \label{eq:def-hat-T_n}
  \hat T_n(s) = \frac{1}{k} \sum_{j=1}^n
  1_{ \{Y_j >
    Y_{n-k:n} (1 + s) \}} \; .
\end{align*}
In this section we consider the \textit{practical} process
$$
\hat e_n^*(s) = \hat T_n(s) - T(s), \qquad s\in [0,\infty) \; .
$$
For the process $\hat e_n^*(\cdot)$, the previous results yield the
following corollary.
\begin{cor} \label{cor:practical} Assume {\rm (H)}, (\ref{eq:lrd}),
  (\ref{eq:Pareto-assumption}), %(\ref{eq:second-order-tightness}),
  (\ref{eq:moment-2alpha+epsilon}) and
  (\ref{eq:second-ordre-unprimitive}). Then $w_n\hat e_n^*$ converges
  weakly in $D([0,\infty))$ to $w - T \cdot w(0)$, i.e.
\begin{itemize}
\item If $\lim_{n\to\infty} k\rho_n^q = 0$ or $\{X_j\}$ is weakly dependent, then
  \begin{equation}
    \label{eq:iid-modified-1}
    \sqrt{k} \hat e_n^* \Rightarrow B \circ T
\end{equation}
where $B$ is the Brownian bridge.
\item If $\lim_{n\to\infty} k\rho_n^q\to\infty$, then
$$
\rho_n^{-q/2} \hat e_n^* \Rightarrow  0\; .
$$
\end{itemize}
\end{cor}
The convergence of $w_n(\hat T_n - T)$ to $w  - T\cdot w(0)$ is
standard. The surprising result is that in the LRD zone the limiting
process is 0, because the limiting process of $w_n(\hat T_n-T_n)$
has a degenerate form, i.e. the limit is the random $L_q$,
multiplied by the deterministic function $T(\cdot)$. In fact, as we
will see below, there is no dichotomy for the process with random
levels, and the rate of convergence of $\hat e_n^*$ is the same as
in the i.i.d. case.

To proceed, we need to introduce a more precise second order
conditions on the distribution function $F_Z$ of $Z$.  Several types
of second order assumptions have been proposed in the literature. We
follow here \cite{drees:1998}.

\paragraph{Assumption (SO)} There exists a bounded non increasing
function $\eta^*$ on $[0,\infty)$, regularly varying at infinity
with index $-\alpha\beta$ for some $\beta\geq0$, and such that
$\lim_{t\to\infty} \eta^*(t)=0$ and there exists a measurable
function $\eta$ such that for $z>0$,
\begin{gather}
  \pr(Z>z) = c z^{-\alpha} \exp\int_1^z \frac{\eta(s)}s \, \d s \; ,
  \label{eq:representation}
  \\
  \exists C>0 \;, \ \ \forall s \geq 0 \; , \ \ |\eta(s)| \leq
  C \eta^*(s) \; .
  \label{eq:borne-eta}
\end{gather}
If~(\ref{eq:representation}) and (\ref{eq:borne-eta}) hold, we will
say that $\bar F_Z$ is second order regularly varying with index
$-\alpha$ and rate function $\eta^*$, in shorthand $\bar F_Z \in
2RV(-\alpha,\eta^*)$.

\begin{thm}\label{thm:practical-1}
  % Consider the case $1-q(1-H)>1/2$.
  Assume {\rm (H)}, (\ref{eq:lrd}),
  (\ref{eq:Pareto-assumption}), {\rm (SO)} with rate function $\eta^*$
  regularly varying at infinity with index $-\alpha\beta$ and there
  exists $\epsilon >0$ such~that
\begin{equation}\label{eq:additional-moment}
0 < \Exp[\sigma^{2\alpha(\beta+1)+\epsilon}(X_1)]<\infty \; .
\end{equation}
If
\begin{equation} \label{eq:negligibility-1a}
  \lim_{n\to\infty} \sqrt k \eta^*(U(n/k)) = 0 \; ,
\end{equation}
then $\sqrt k \hat e_n^*$ converges weakly in $D([0,\infty))$ to $B
\circ T$, where $B$ is the Brownian bridge (regardless of the
behaviour of $k\rho_n^q$).
\end{thm}
\noindent{\em Remark}. \
 The additional moment condition (\ref{eq:additional-moment}) ensures that the
 distribution of $Y$ satisfies a second order condition.  See Section
 \ref{sec:s-o-c} for more details. It is also used in a proof of tightness
 argument (see (\ref{eq:uniform-conv}) below).

  The behaviour described in Theorem \ref{thm:practical-1} is quite
  unexpected, since the process with {\it estimated} levels
  $Y_{n-k:n}$ has a faster rate of convergence than the one with
  the deterministic levels $u_n$. A similar phenomenon was observed in
  the context of LRD based empirical processes with estimated
  parameters. We refer to \cite{Kulik2008b} for more details.

\subsection{Tail index estimation}
\label{sec:Hill}

A natural application of the asymptotic result for the tail empirical
process $\hat e_n^*$ is the asymptotic normality of the Hill
estimator of the extreme value index $\gamma$ defined by
$$
\hat\gamma_n = \frac{1}{k}\sum_{i=1}^k \log \left(
  \frac{Y_{n-i+1:n}}{Y_{n-k:n}} \right) = \int_0^\infty \frac{\hat
  T_n(s)}{1+s} \, \d s \; .
$$
Since $\gamma = \int_0^\infty (1+s)^{-1}T(s) \, \d s$, we have
\begin{align*}
  \hat\gamma_n - \gamma = \int_0^\infty \frac{\hat e_n^*(s)}{1+s} \,
  \d s \; .
\end{align*}
Thus we can apply Theorem~\ref{thm:practical-1} to obtain the
asymptotic distribution of the Hill estimator.
\begin{cor}
  \label{cor:Hill}
  Under the assumptions of Theorem~\ref{thm:practical-1}, $\sqrt k
  (\hat\gamma_n - \gamma)$ converges weakly to the centered Gaussian
  distribution with variance $\gamma^2$.
\end{cor}
It is known that the above result gives the best possible rate of
convergence for the Hill estimator (see \cite{drees:1998}). The
surprising result is that it is possible to achieve the i.i.d. rates
regardless of $H$.

\subsection{Second order conditions}
\label{sec:s-o-c} Whereas the transfer of the tail index of $Z$ to
$Y$ is well known, the transfer of the second order property seems
to have been less investigated. We state this in the next
proposition, as well as the rate of convergence of $T_n$ to $T$ and
$G_n$ to $G \times T$.

\begin{prop}
   \label{prop:transfert-second-ordre}
   %If~(\ref{eq:Pareto-assumption}),~(\ref{eq:representation})
   %and~(\ref{eq:borne-eta}) hold, if
   If $\bar F_Z \in 2RV(-\alpha,\eta^*)$, where $\eta^*$ is regularly varying at
   infinity with index $-\alpha\beta$, for some $\beta\geq0$, and if
   \begin{align}
     \esp[\sigma^{\alpha(\beta+1)+\epsilon}(X)] < \infty \; ,
   \end{align}
   for some $\epsilon>0$, then
   % $\pr(Y>y) = y^{-\alpha} \ell_2(y)$ with $\ell_2 \in
   % 2RV(0,\eta^*)$, i.e.
   $\bar F \in 2RV(-\alpha,\eta^*)$, %(\ref{eq:second-order-tightness})
   % and~(\ref{eq:second-order-tightness-Z})
   %holds
   and
   \begin{gather}
     \|T_n - T\|_\infty = O(\eta^*(u_n)) \; . \label{eq:rate-T_n}
   \end{gather}
   Moreover, for any $p\geq1$ such that
   $p\alpha(\beta+1)<\alpha(\beta+1)+\epsilon$,
   % If moreover
   % \begin{align}
   %   \esp[\sigma^{p\alpha(\beta+1)+\epsilon}(X)] < \infty \;
   % \end{align}
   % for some $p\geq1$, then
   \begin{gather}
     \esp \left[ \sup\nolimits_{s\geq0} |G_n(X,s) - \sigma^\alpha(X)
       T(s) |^p \right] = O(\eta^*(u_n)^p) \; .
     \label{eq:rate-G_n}
     \end{gather}

 \end{prop}

\paragraph{Examples}

The most commonly used second order assumption is that $\eta^*(s) =
O(s^{-\alpha\beta})$ for some $\beta>0$. Then
\begin{gather}
  \label{eq:Pareto-second-order}
  \bar F_Z(x) = cx^{-\alpha}(1 + O(x^{-\alpha\beta})) \ \ \mbox{ as }
  x\to\infty \; ,
\end{gather}
for some constant $c>0$.  Then, $\|T_n-T\|_\infty =
O((k/n)^{\beta})$, and the second order
condition~(\ref{eq:second-ordre-unprimitive}) becomes
\begin{equation}
  \label{eq:negligibility-1}
  \lim_{n\to\infty} k  \left(\frac{k}{n}\right)^{2\beta} = 0 \; ,
  \ \mbox{ if } \ \lim_{n\to\infty} k\rho_n^q = 0
\end{equation}
and
\begin{equation}\label{eq:negligibility-2}
  \lim_{n\to\infty}  \rho_n^{-q} \left(\frac{k}{n}\right)^{2\beta} =  0
  \ \mbox{ if }
  \ \lim_{n\to\infty} k\rho_n^q = \infty \; .
\end{equation}
Condition~(\ref{eq:negligibility-1}) holds if both $k \ll
n^{(2\beta)/(2\beta+1)}$ and $k \ll n^{2(1-H)}$. The central limit
theorem with rate $\sqrt k$ holds if $k\asymp n^\gamma$ with
$$
\gamma < 2(1-H) \vee \frac{2\beta}{2\beta+1} \; .
$$
Condition~(\ref{eq:negligibility-2}) holds if $n^{2(1-H)} \ll k \ll
n^{1-(1-H)/\beta}$. This may happen only if
\begin{align*}
  \beta > \frac{1-H}{2H-1}
\end{align*}
or equivalently
$$
1 > H > \frac{1+\beta}{2\beta+1} \; .
$$
As $\beta\to0$, only for very long memory processes (i.e. $H$ close
to 1) will the LRD zone be possible.

The extreme case is the case $\beta=0$, i.e. $\eta^*$ slowly
varying. For instance, if $\eta^*(x) = 1/\log(x)$ (for $x$ large),
then the tail $\bar F(x) = x^{-\alpha} \log(x)$ belongs to
$2RV(-\alpha,\eta^*)$ and $U(t) \sim
\{t\log(t)/\alpha\}^{-1/\alpha}$. The second order
condition~(\ref{eq:second-ordre-unprimitive}) holds if
\begin{align*}
  k^{1/2} \log^{-1}(n) \to 0 \; .
\end{align*}
If this condition holds, then $k \rho_n^q\to0$ for any $H>1/2$ and the LRD zone
never arises, because the LRD term in the
decomposition~(\ref{eq:decomposition-sv}) is always dominated by the bias.

\section{Numerical results}\label{sec:numstudies}
We conducted some simulation experiments to illustrate our results.  We used
{\tt R} functions {\tt HillMSE()} and {\tt HillPlot} available on the authors
webpages.

Our first experiment deals with the Mean Squared Error.
\begin{enumerate}
\item Using {\tt R-fracdiff} package we simulated fractional Gaussian noises
  sequences $\{X_i(d)\}$ with parameters $d=0, 0.2, 0.4, 0.45$. Here, $d=H-1/2$, so that $d=0$ corresponds to the case of an
  i.i.d. sequence.
\item We simulated $n=1000$ i.i.d. Pareto random variables $Z_i$ with
  parameters $\alpha=1$ and $2$.
\item We set $Y_i(d)=\exp(X_i(d))Z_i$.
\item Hill estimator was constructed for different number of extremes.
\item This procedure was repeated 10000 times.
\item The results are displayed on Figure \ref{fig:1}, for $\alpha=1$ and
  $\alpha=2$, respectively. On each plot, we visualise Mean Square Error (with
  the true centering) w.r.t. the number of extremes. Solid lines represent
  different LRD parameters: black for $d=0$, blue for $d=0.2$, red
  for $d=0.4$ and green for $0.45$.
\end{enumerate}
\begin{figure}
  \begin{center}
    \includegraphics [width=0.45\textwidth,height=.3\textheight]{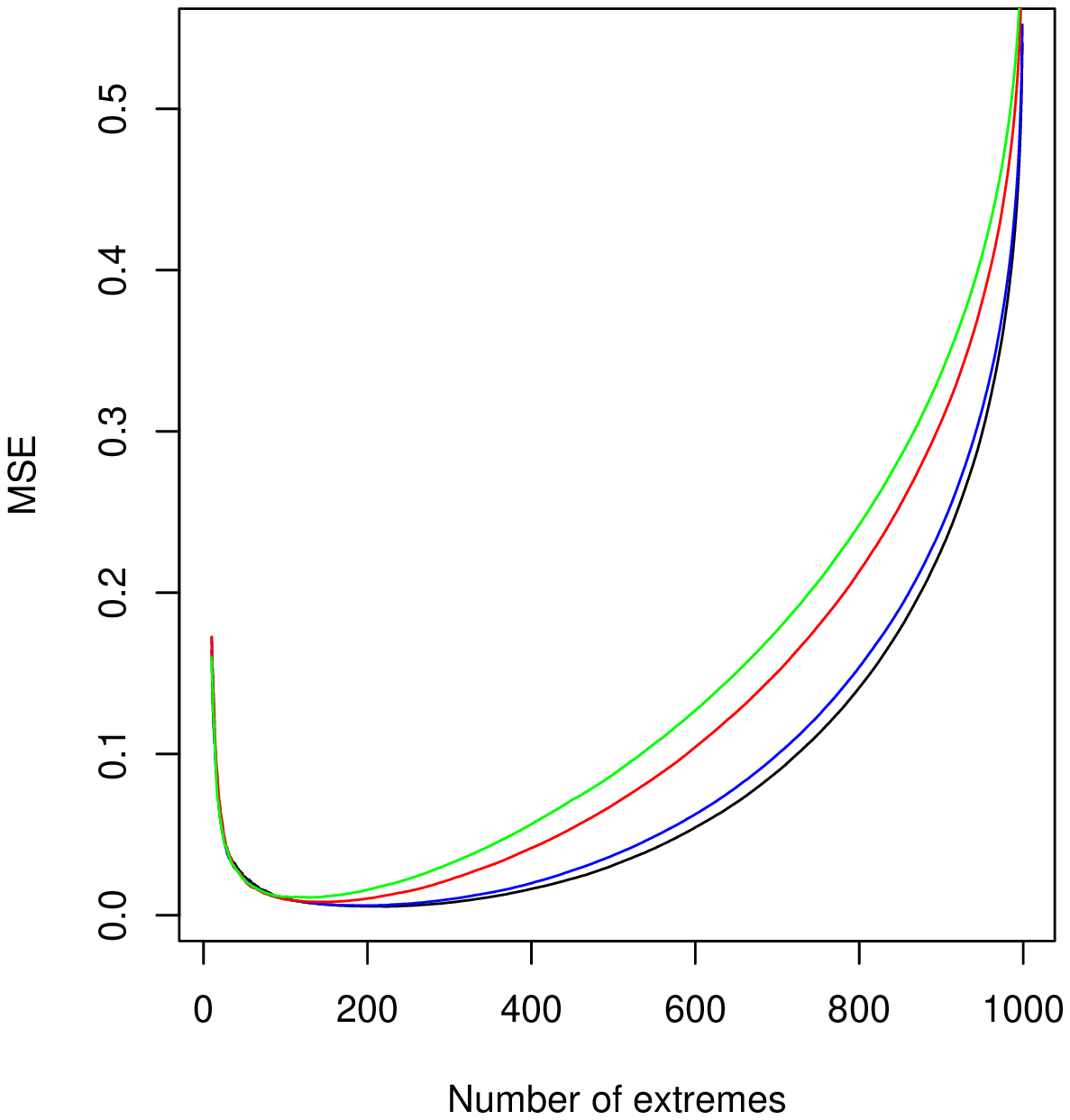}
    \includegraphics[width=0.45\textwidth,height=.3\textheight]{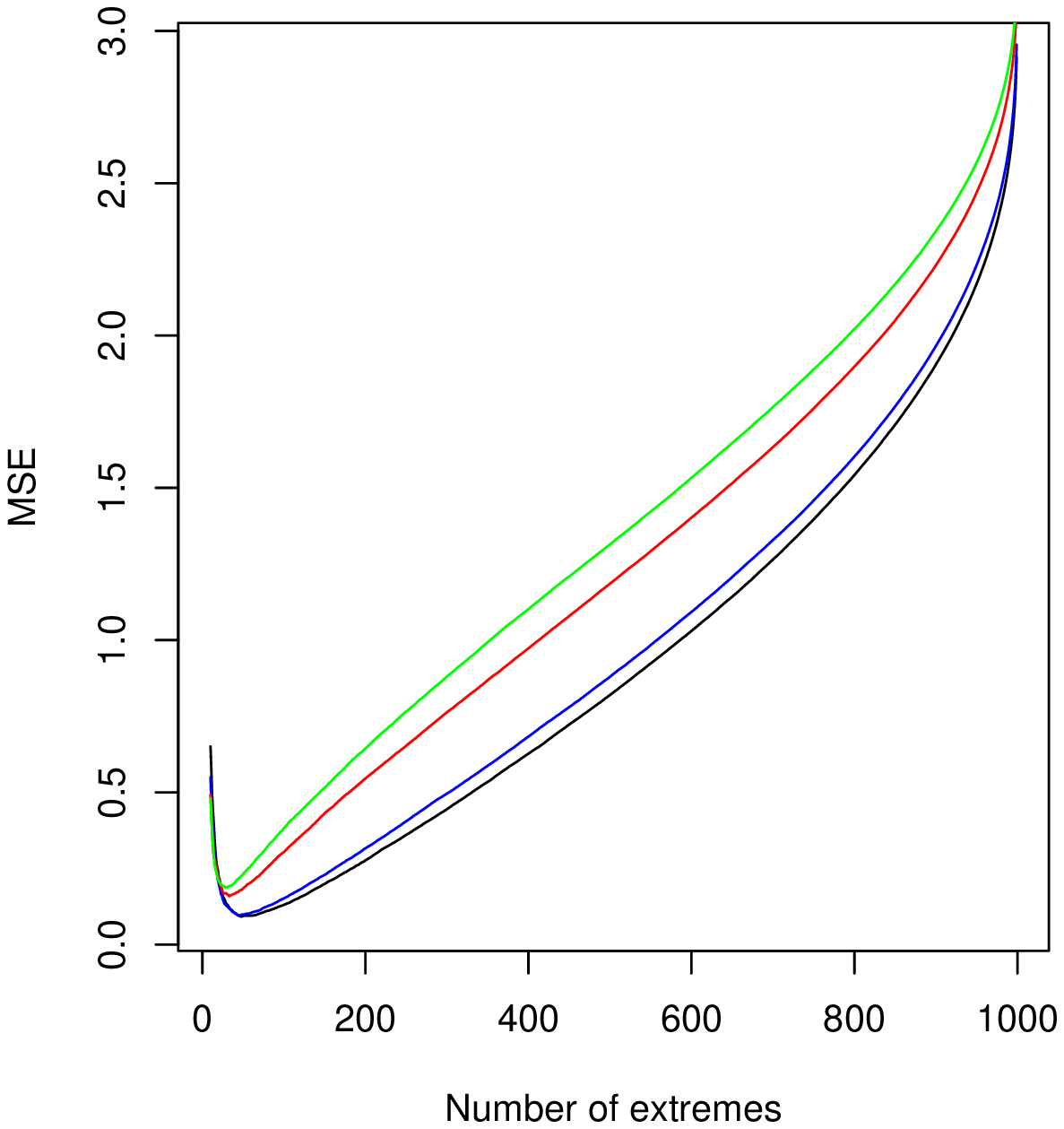}
    \caption{\small MSE: $\alpha=1$ (left panel), $\alpha=2$ (right panel);
    color codes: black - $d=0$, blue - $d=0.2$, red - $d=0.4$, green - $0.45$} \label{fig:1}
\end{center}
\end{figure}
We note that for $\alpha=1$, when a small number of extreme order
statistics $k$ is used to build the Hill estimator, there is not
much influence of the LRD parameter, and in particular the MSE is
minimal for more or less the same values of $k$ through all the
range of values of $d$. This is in accordance with our theoretical
results. For $\alpha=2$, the influence of the memory parameter is
more significant. These two features can be interpreted.  First, it
seems natural that the long memory effect appears when a greater
number of extreme order statistics is used, since our result is of
an asymptotic nature. For a small number of extremes the i.i.d. type
of behaviour dominates (see $R_n(\cdot)$ in
(\ref{eq:decomposition-sv})), so the asymptotic result is seen; for
a larger number of extremes, the long memory term $S_n$ in
(\ref{eq:decomposition-sv}) starts to dominate.  For an extremely
large number of order statistics (i.e. $k\asymp n$), the bias
dominates. The influence of $\alpha$ on the quality of the
estimation is twofold. On one hand, the asymptotic variance of the
Hill estimator is $\alpha^2$, so that the MSE increases with
$\alpha$. Also, for very small values of $\alpha$, the peaks
observed are extremely high and completely overshadow the effect of
long memory.

Next, we show Hill plots for several models, since in practice one usually
deals with just a single realization.
\begin{enumerate}
\item We consider the model $Y_i=\exp(\tau X_i)Z_i$, where $\{X_i\}$ is as
  above a fractional Gaussian noise and $\tau=0.05$ or $2$.
\item We simulated $n=1000$ i.i.d. Pareto random variables $Z_i$ with
  parameter $\alpha=2$.
\item We simulated fractional Gaussian noise sequences $\{X_i\}$ with
  parameters $d=0$ (i.i.d. case), 0.2, 0.4, $0.45$.
\item The estimators are plotted on Figures \ref{fig:2} and \ref{fig:3}. The
  left panel corresponds to the Hill estimator for iid Pareto random variables
  $\{Z_i\}$, and the right one for the long memory stochastic volatility
  process $\{Y_i\}$. Recall that the $Y_i$ are dependent asympotically Pareto
  random variables, so that there are two sources of bias for the Hill
  estimator.
\end{enumerate}
\begin{figure}
  \begin{center}
    \includegraphics[width=1\textwidth,height=.3\textheight]{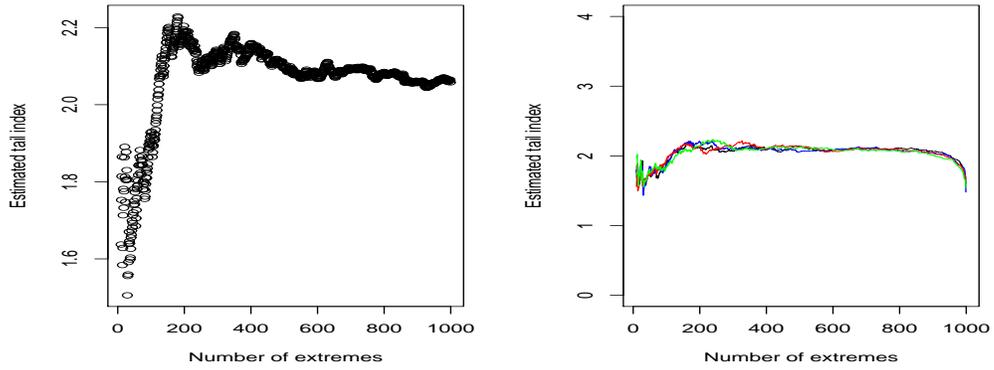}
 \caption{\small Hill estimator: $\alpha=2$ and Pareto iid (left panel),
   $\tau=0.05$ (right panel); color codes: black - $d=0$, blue - $d=0.2$, red - $d=0.4$, green - $0.45$}
\label{fig:2}
\end{center}
\end{figure}
\begin{figure}
    \begin{center}
      \includegraphics[width=1\textwidth,height=.3\textheight]{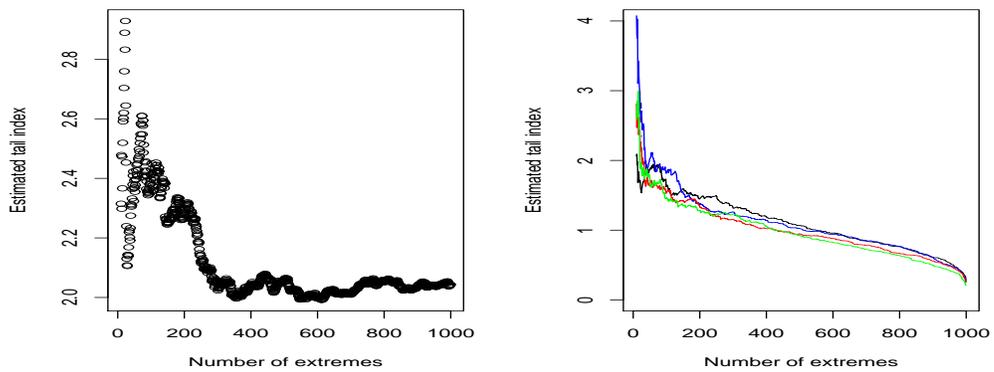}
      \caption{\small Hill estimator: $\alpha=2$ and Pareto iid (left panel),
        $\tau=1$ (right panel); color codes: black - $d=0$, blue - $d=0.2$, red - $d=0.4$, green - $0.45$}
\label{fig:3}
\end{center}
\end{figure}
We may observe that for a small volatility parameter $\tau$ there is not too
much difference between the two plots. However, if $\tau$ becomes bigger, the
estimation with a large number of extremes is completely inappropriate if
$d>0$, though without much influence of the strength of the dependence
(i.e. increase of $d$) on this degradation. The reason is that the second order
condition satisfied by the stochastic volatility model yields the same rate of
convergence as in the i.i.d. case, but an increase in the variance of the
Gaussian process $\{X_t\}$ entails a bigger bias in finite sample.

\section{Proofs}
\subsection{Gaussian long memory sequences}
\label{sec:LRD-Gaussian} Recall that each function $G(\cdot)$ in
$L^2(\d\mu)$, with $\mu(\d x) = (2\pi)^{-1/2} \exp(-x^2/2) \, \d x$
can be expanded as
$$
G(X) = \esp[G(X)] + \sum_{m=1}^{\infty}\frac{J(m)}{m!}H_m(X) \; ,
$$
where $J(m)=\Exp[G(X)H_m(X)]$ and $X$ is a standard Gaussian random
variable. Recall also that the smallest $q\ge 1$ such that
$J(q)\not=0$ is called the Hermite rank of $G$. We have
\begin{equation} \label{eq:Rozanov}
  \Exp[G(X_0)G(X_k)] = \esp[G(X_0)] + \sum_{m=q}^{\infty} \frac{J^2(m)}{m!} \,
\rho_k^m \; ,
\end{equation}
where $\rho_k=\Cov(X_0,X_k)$.  Thus, the asymptotic behaviour of
$\Exp[G(X_0)G(X_k)]$ is determined by the leading term $\rho_n^q$.
In particular, if $1-q(1-H)>1/2$, which implies that $n^2\rho_n^q
\to\infty$,
\begin{gather}\label{eq:Var-sum}
  \Var \left( \sum_{j=1}^nG(X_j) \right) \sim \frac{J^2(q)}{q!}  \;
  \frac{n^2 \rho_n^q}{1-2q(1-H)}
\end{gather}
and
\begin{gather}
  \label{eq:lim-sums}
  \frac{1}{n\rho_n^{q/2}} \sum_{j=1}^nG(X_j) \convdistr J(q) L_q \; ,
\end{gather}
where
\begin{align} \label{eq:def-L_q}
  L_q = (q!(1-2q(1-H))^{-1/2} Z_{H,q}(1)
\end{align}
and $Z_{H,q}$ is the so-called Hermite or Rosenblatt process of
order $q$, defined as a $q$-fold stochastic integral
\begin{align*}
  Z_{H,q}(t) =\int_{-\infty}^\infty \dots \int_{-\infty}^\infty
  \frac{\mathrm e^{\mathrm i t(x_1+\cdots+x_q)}-1}{x_1+\cdots+x_q} \,
  \prod_{i=1}^q x_i^{-H+1/2} \, W(\d x_1) \dots W(\d x_q) \; ,
\end{align*}
where $W$ is an independently scattered Gaussian random measure with
Lebesgue control measure.  For more details, the reader is referred
to \cite{Taqqu2003}. On the other hand, if $1-q(1-H)<1/2$ or
$\{X_j\}$ is weakly dependent, then
\begin{equation}\label{eq:lim-sums-iid}
\frac{1}{\sqrt{n}}\sum_{j=1}^nG(X_j)\convdistr {\cal
N}(0,\Sigma_0^2),
\end{equation}
where
$\Sigma_0^2=\Var(G(X_0))+2\sum_{j=1}^{\infty}\Cov(G(X_0),G(X_j))<\infty$.

We will also need the following variance inequalities of \cite{arcones:1994}:
\begin{itemize}
\item If $1-q(1-H)>1/2$, then for any function $G$ with
  Hermite rank $q$ ,
\begin{align}
  \label{eq:variance-inequality-lrd}
  \var\left( n^{-1}\sum_{j=1}^n G( X_j) \right) \leq C \rho_n^q \;
  \var(G(X_1)) \; .
\end{align}
\item If $1-q(1-H)<1/2$, then for any function $G$ with
  Hermite rank $q$ ,
\begin{align}
  \label{eq:variance-inequality-weak-dependence}
  \var\left(n^{-1} \sum_{j=1}^n G(X_j) \right) \leq C n^{-1} \; \var(G(X_1))
  \; .
\end{align}
\end{itemize}
In all these cases, the constant $C$ depends only on the Gaussian
process $\{X_j\}$ and not on the function $G$. The
bounds~(\ref{eq:variance-inequality-lrd})
and~(\ref{eq:variance-inequality-weak-dependence}) are Equation~3.10
and~2.40 in \cite{arcones:1994}, respectively.
\subsection{Decomposition of the tail empirical process}
The main ingredient of the proof of our results will be the
following decomposition. Let ${\cal X}$ be the $\sigma$-field
generated by the Gaussian process $\{X_n\}$.
\begin{align}
  e_n(s) & = \frac{1}{n\bar F(u_n)} \sum_{j=1}^n \left\{
    1_{\{Y_j>(1+s)u_n\}} - \pr(Y_j>(1+s)u_n|{\cal X}) \right\} \nonumber   \\
  & \ \ \ + \frac{1}{n\bar F(u_n)} \sum_{j=1}^n \left\{\pr(Y_j>(1+s)u_n|{\cal
      X}) - \bar F(u_n)\right\} \nonumber  \\
  & = :R_n(s) + S_n(s) \; .
  \label{eq:decomposition-sv}
\end{align}
Conditionally on $\mathcal X$, $R_n$ is the sum of independent
random variables, so it will be referred to as the i.i.d. part; the
term $S_n$ is the partial sum process of a subordinated Gaussian
process, so it will be referred to as the LRD part.

\subsection{Proof of Theorem~\ref{thm:general}}
We first give a heuristic behind the dichotomous behaviour in
Theorem \ref{thm:general}. Then, we prove convergence of the finite
dimensional distributions of the i.i.d. and LRD parts. Finally, we
prove tightness and asymptotic independence.
\subsubsection{Heuristic}\label{sec:heuristic}
To present some heuristic, let us compute covariance of the tail
empirical process. We have
\begin{align*}
  \Cov(\tilde T_n(s),\tilde T_n(t)) & = \frac{1}{n\bar
    F^2(u_n)}\Cov(1_{\{Y_1>u_n(1+s)\}},1_{\{Y_1>u_n(1+t)\}}) \\
  & \ \ \ + \frac{2}{n^2\bar F^2(u_n)}\sum_{j=1}^{n-1}(n-j)\Cov
  (1_{\{Y_1>u_n(1+s)\}},1_{\{Y_{j+1}>u_n(1+t)\}} \; .
\end{align*}
Recall (\ref{eq:cond-tail}). If
$\Exp[\sigma^{\alpha+\epsilon}(X_1)]<\infty$ holds, we apply
Breiman's Lemma to both nominator and denominator to get
\begin{multline*}
  \lim_{n\to\infty}\frac{\Cov(1_{\{Y_1>u_n(1+s)\}},1_{\{Y_1>u_n(1+t)\}})}{\bar
    F(u_n)} \\
  = \lim_{n\to\infty}\frac{\Exp[\sigma^{\alpha}(X_1)]P(Z_1>u_n(1+s)\vee
    u_n(1+t))}{\Exp[\sigma^{\alpha}(X_1)]P(Z_1>u_n)}=T(s\vee t) \; .
\end{multline*}
Furthermore, if
$\Exp[\sigma^{\alpha+\epsilon}(X_1)\sigma^{\alpha+\epsilon}(X_{j+1})]<\infty$
holds (which is guaranteed by (\ref{eq:moment-2alpha+epsilon})),
then a generalization of Breiman's Lemma yields
\begin{align*}
  \lim_{n\to\infty} & \frac{\Cov(1_{\{Y_1>u_n(1+s)\}},
    1_{\{Y_{j+1}>u_n(1+t)\}})} {\bar F^2(u_n)}  \\
  & = \lim_{n\to\infty} \frac{P(Y_1>u_n(1+s), Y_{j+1}>u_n(1+t))} {\bar
    F^2(u_n)} - T(s) T(t)  \\
  & = T(s) T(t) \left( \frac{ \Exp[\sigma^{\alpha}(X_1)
      \sigma^{\alpha}(X_{j+1})]} {\Exp[\sigma^{\alpha}(X_1)]
      \Exp[\sigma^{\alpha}(X_{j+1})]} - 1 \right) \; .
\end{align*}
Therefore, for fixed $s$ and $t$, using (\ref{eq:Var-sum}) in the
case $q(1-H)<1/2$, we obtain
\begin{align*}
  \Cov&(\tilde T_n(s), \tilde T_n(t)) \\
  & = (1+o(1)) \frac{T(s\vee t)} {n \bar F(u_n)}    \\
  & \ \ \ + (1+o(1))  \frac{T(s) T(t)}{ \Exp^2[\sigma^{\alpha}(X_1)]} \frac{1}{n}
  \sum_{j=1}^{n-1}\left(1-\frac jn \right)
  \Cov(\sigma^{\alpha}(X_1),\sigma^{\alpha}(X_{j+1}))
  \\
  &= (1+o(1)) \left( \frac{T(s\vee t)}{n\bar F(u_n)} + \frac{ T(s) T(t)
      J^2(q)\rho_n^q}{q!(1-2q(1-H))\Exp^2[\sigma^{\alpha}(X_1)]}\right).
\end{align*}
In particular, setting $s=t$, then we conclude that the
normalization factor for $e_n(\cdot)$ should be $\sqrt{n\bar
F(u_n)}$ or $\rho_n^{-q/2}$ depending whether $n \bar F(u_n)
\rho_n^q \to 0$ or $n \bar F(u_n) \rho_n^q \to \infty$ holds. The
asymptotic variance also suggests the form of limiting distributions
in Theorem \ref{thm:general}.
\subsubsection{Finite dimensional limits}
\label{sec:fidi}
Let $\convdistr$ denote weak convergence of finite
dimensional distributions.  It will be shown in Section
\ref{sec:iid-limit} and \ref{sec:LRD-limit}, respectively, that for
each $m\ge 1$ and $s_l\in [0,\infty)$, $l=1,\ldots,M$,
$s_1<\cdots<s_M$,
\begin{multline}
  \sqrt{n\bar F(u_n)} \left(R_n(s_1),R_n(s_l) - R_n(s_{l-1}), l = 2,
    \ldots, M \right)  \\
  \convdistr \left( {\cal N}(0,T(s_1)),{\cal N}(0,T(s_l) - T(s_{l-1})),
    l=2, \ldots, M \right) \; ,
  \label{eq:iid-limit}
\end{multline}
where the normal random variables are independent, and
\begin{equation}\label{eq:LRD-limit}
  \rho_n^{-q/2} (S_n(s_1),\ldots,S_n(s_M))
  \convdistr \frac{J(q)}{\Exp[\sigma^{\alpha}(X_1)]} (T(s_1),\ldots,T(s_M))
  L_q \; ,
\end{equation}
if $1-q(1-H)>1/2$.  On the other hand, if $1-q(1-H)<1/2$, then the
second term $S_n(\cdot)$ is of smaller order than the first one,
$R_n(\cdot)$.

\subsubsection*{The i.i.d. limit}\label{sec:iid-limit}
Define
\begin{align*}
  L_{n,j}(x,s) & =   1_{\{\sigma(x) Z_j>(1+s)u_n\} } - \pr(
  \sigma(x) Z_1 >(1+s)u_n) \; .
\end{align*}
Then
\begin{align*}
  R_n(s) = \sum_{j=1}^n L_{n,j}(X_j,s) \; .
\end{align*}
Set $L_{n,j}(x)=L_{n,j}(x,0)$ and $V_n^{(m)}(x)=\esp[L_{n,j}^m(x)]$.
Note that $\Exp[V_n^{(1)}(X_j)]=0$ and
\begin{align*}
  V_n^{(2)}(x) & = \pr (\sigma(x) Z_1 > u_n) - \pr^2( \sigma(x) Z_1 >
  u_n ) \; .
\end{align*}
Let $R_n:=R_n(0)$. Therefore, for fixed $t$,
\begin{align}
  & \log \Exp\left[\mathrm e^{\mathrm i t\sqrt{n\bar F(u_n)}R_n}|{\cal
      X} \right] \nonumber \\
  & = \sum_{j=1}^n \log \Exp \left[ \exp\left(\frac{\mathrm
        it}{\sqrt{n \bar F(u_n)}} \{  1_{\{Y_j>u_n\}} -
      \pr(Y_j>u_n \mid {\cal X}) \} \right) \mid {\cal X} \right]
  \nonumber  \\
  & = \sum_{j=1}^n \log\Exp \left[ 1 - \frac{  it}{\sqrt{n\bar
        F(u_n)}} L_{n,j}(X_j) - \frac{t^2}{2n \bar F(u_n)}
    L_{n,j}^{2}(X_j) + L_{n,j}^{3}(X_j) O\left(\frac{1}{(n \bar
        F(u_n))^{3/2}} \right)
    \mid {\cal X} \right] \nonumber  \\
  & = \frac{-t^2}{2n\bar
    F(u_n)}\sum_{j=1}^nV_n^{(2)}(X_j)+o\left(\frac{1}{n\bar
      F(u_n)}\right)\sum_{j=1}^nV_n^{(2)}(X_j)+O\left(\frac{1}{(n\bar
      F(u_n))^{3/2}}\right) \sum_{j=1}^n |V_n^{(3)}(X_j)| \; .
  \label{eq:decomp-chf}
\end{align}
We will show that
\begin{equation}\label{eq:lln}
\frac{1}{n\bar F(u_n)}\sum_{j=1}^nV_n^{(2)}(X_j)\convprob 1,
\end{equation}
given that $\Exp[\sigma^{\alpha+\delta}(X_1)]<\infty$.  This also
shows that the second term in (\ref{eq:decomp-chf}) is negligible.
Furthermore, since for sufficiently large $n$ and $\delta>0$ (cf.
(\ref{eq:bound-1})),
$$
|V_n^{(3)}(x)| \le C \pr (\sigma(x) Z_1 > u_n ) \le C ( \sigma(x)
\vee 1)^{\alpha+\delta} P(Z_1 > u_n) \; ,
$$
the expected value of the last term in (\ref{eq:decomp-chf}) is
$$
O\left(\frac{nP(Z_1>u_n)}{(n\bar F(u_n))^{3/2}}\right)\Exp[1\vee
\sigma^{\alpha+\delta}(X_1)]  \; .
$$
Consequently, the last term in~(\ref{eq:decomp-chf}) converges to 0
in $L^1$ and in probability.  Therefore, on account of
(\ref{eq:lln}) and the negligibility, we obtain,
\begin{equation}
  \label{eq:chf-1}
  \log\Exp \left[\mathrm e^{\mathrm it\sqrt{n\bar F(u_n)}R_n}|{\cal X} \right]
  \convprob -t^2/2
\end{equation}
and from bounded convergence theorem we conclude
\eqref{eq:iid-limit} (for $M=1$ and $s=0$).  It remains to prove
(\ref{eq:lln}). By Lemma~\ref{lem:convergence-uniforme-hermite}, for
each $j\ge 1$, $G_n(X_j,s)$ converges in probability and in $L^1$ to
$\sigma^\alpha(X_j)$.  Therefore,
\begin{equation}\label{eq:1}
  \lim_{n\to\infty}   \esp \left[ \left|
      \frac{1}{n} \sum_{j=1}^n  \frac{\pr(\sigma(X_j) Z_1 >
        u_n \mid \mathcal X)}{\pr(Z_1>u_n)} - \sigma^{\alpha}(X_j) \right| \right]
  = 0 \; .
\end{equation}
Next, since $\sigma^{\alpha}(X_j)$, $j\ge 1$, is ergodic, we have
\begin{equation}\label{eq:2}
  \frac{1}{n} \sum_{j=1}^n \sigma^{\alpha}(X_1) \convprob
  \esp[\sigma^\alpha(X_1)] \; .
\end{equation}
Thus, (\ref{eq:1}), (\ref{eq:2}) and Breiman's Lemma yields
\begin{align} \label{eq:bout-1}
  \frac{1}{n\bar F(u_n)} \sum_{j=1}^n \pr(\sigma(X_j) Z_1 > u_n \mid
  \mathcal X) \convprob 1 \; .
\end{align}
Write now
\begin{align*}
  \frac{1}{n \bar F(u_n)} \sum_{j=1}^n V_n^{(2)}(X_j) & =  1 + o_P(1) +
  \frac{1}{n\bar F(u_n)} \sum_{j=1}^n \pr^2 (\sigma(X_j) Z_1 > u_n
  \mid \mathcal X) \; .
\end{align*}
By Lemma~\ref{lem:convergence-uniforme-hermite}, we have, for some
$\delta>0$ small enough,
\begin{multline}
  \frac{1}{n\bar F(u_n)} \sum_{j=1}^n \pr^2 (\sigma(X_j) Z_1 > u_n
  \mid \mathcal X) \\
  \leq C \pr(Z>u_n) \frac1n \sum_{j=1}^n (\sigma(X_j)\vee1)^{2\alpha+\delta}
  \convprob 0 \; .
\end{multline}
This proves (\ref{eq:lln}) and (\ref{eq:iid-limit}) follows with
$M=1$ and $s_1=0$. The case of a general $M\ge 1$ is obtained
analogously.

\subsubsection*{Long memory limit}\label{sec:LRD-limit}
Recall the definition (\ref{eq:function-gn}) of $G_n(\cdot,s)$ and
that $G(x)=\sigma^{\alpha}(x)$.  Define
$$J_{n}(m,s)=\Exp[H_m(X_1)G_n(X_1,s)], \quad
J(m)=\Exp[H_m(X_1)G(X_1)],$$ the Hermite coefficients of
$G_n(\cdot,s)$ and $G(\cdot)$, respectively.  Let $q$ be the Hermite
rank of $G(\cdot)$.  We write (recall Assumption (H)),
\begin{eqnarray}
  \lefteqn{\sum_{j=1}^n (G_n(X_j,s)  - \Exp[G_n(X_j,s)])} \nonumber   \\
  & =& \sum_{j=1}^n \sum_{m=q}^{\infty} \frac{T(s)J(m)}{m!} H_m(X_j) +
  \sum_{j=1}^n \sum_{m=q}^{\infty} \frac{J_n(m,s)-T(s)J(m)}{m!}
  H_m(X_j)  \nonumber \\
  & = :& T(s) S_n^* + \tilde S_n(s) \; ,
  \label{eq:expansion}
\end{eqnarray}
with $S_n^* = \sum_{j=1}^n G(X_j)$.  On account of Rozanov's
equality (\ref{eq:Rozanov}), we have that the variance of the second
term is
\begin{align}
  \var(\tilde S_n(s) ) & = \sum_{i,j=1}^n \sum_{m=q}^{\infty}
  \frac{(J_n(m,s)-T(s)J(m))^2}{m!}  \Cov^m(X_i,X_j) \nonumber
  \\
  & \leq \sum_{i,j=1}^n |\Cov^{q}(X_i,X_j)| \sum_{m=q}^{\infty}
  \frac{(J_n(m,s)-T(s)J(m))^2}{m!}  \nonumber
  \\
  & = \|G_n(\cdot,s)-T(s)G(\cdot)\|_{L^2(d\mu)}^2 \sum_{i,j=1}^n
  |\Cov^{q}(X_i,X_j)| \nonumber \\
  & \leq C n^2 \rho_n^q \, \|G_n(\cdot,s)-T(s)G(\cdot)\|_{L^2(d\mu)}^2
  \; .  \label{eq:control-Hermite}
\end{align}
Since $\esp[\sigma^{2\alpha+\delta}(X)]<\infty$, by
Lemma~\ref{lem:convergence-uniforme-hermite}, $G_n(\cdot,s)$
converges $T(s)G(\cdot)$ in $ L^2(d\mu)$, uniformly with respect to
$s$.
We conclude that the second term on the right handside of
(\ref{eq:expansion}) is $o_P\left(n\rho_n^{q/2}\right)$, i.e. it is
asymptotically smaller than the first term.
Furthermore,
\begin{equation}\label{eq:claim-2}
  S_n(s) = \frac{P(Z_1>u_n)}{n\bar F(u_n)}
  \sum_{j=1}^n\left(G_n(X_j,s)-\Exp[G_n(X_j,s)]\right),
\end{equation}
so that via (\ref{eq:lim-sums}) and (\ref{eq:bound-2})
\begin{equation}\label{eq:chf-2}
   \rho_n^{-q/2} S_n(s) \convdistr
  \frac{J(q)T(s)}{\Exp[\sigma^{\alpha}(X_1)]} L_q \; ,
\end{equation}
if $1-q(1-H)>1/2$.  Consequently, (\ref{eq:LRD-limit}) holds for
$M=1$. The multivariate case follows immediately.  On the other
hand, if $1-q(1-H)<1/2$, then via (\ref{eq:lim-sums-iid}) and
(\ref{eq:bound-2}),
$$
\sqrt{n} \sup_{s\in [0,1]} S_n(s) \convdistr
\frac{1}{\Exp[\sigma^{\alpha}(X_1)]}{\cal N}(0,\Sigma_0^2) \; ,
$$
which proves negligibility with respect to the term $R_n(\cdot)$.

\subsubsection{Asymptotic independence}
In this section we prove asymptotic independence of $R_n(\cdot)$ and
$S_n(\cdot)$. We will carry out a proof for the joint characteristic
function of $(R_n,S_n)=(R_n(0),S_n(0))$. Extension to multivariate
case is straightforward. On account of (\ref{eq:chf-1}),
(\ref{eq:chf-2}) and the bounded convergence theorem, we have
\begin{eqnarray*}
 \lefteqn{ \Exp \left[ \exp \left\{ \mathrm i s \sqrt{n\bar F(u_n) } R_n +
      \mathrm it \rho_n^{-q/2} S_n \right\} \right]}\\
       & =& \Exp \left[
    \Exp [ \exp\{\mathrm i s \sqrt{n\bar F(u_n)} R_n\} \mid {\cal X} ]
    \exp \left( \mathrm i t \rho_n^{-q/2} S_n    \right) \right] \\
  & \to & \exp(-s^2/2) \, \psi_{L_q}
  \left(\frac{J(q)}{\Exp[\sigma^{\alpha}(X_1)]} \, t \right) \
  \mbox{ as } n \to \infty \; ,
\end{eqnarray*}
where $\psi_{L_q}(\cdot)$ is the characteristic function of $L_q$.
This proves asymptotic independence.

\subsubsection{Tightness}
\label{sec:tightness-deterministic-level} In order to prove the
tightness in $\mathcal D([0,\infty))$ endowed with Skorokhod's $J_1$
topology of the sequence of processes $R_n':=\sqrt{n\bar
F(u_n)}R_n$, we apply the tightness criterion of \cite[Theorem
15.4]{billingsley:1968}. We must prove that for each $A>0$ and
$\epsilon>0$,
\begin{align}
  \label{eq:bill-criterion-15.4}
  \lim_{\delta\to0} \limsup_{n\to\infty}
\pr(w_A^{\prime\prime}(R_n',\delta)>\epsilon) = 0 \; ,
\end{align}
where for any function $g \in\mathcal D([0,\infty))$,
\begin{align*}
  w_A^{\prime\prime}(g,\delta) = \sup_{0\leq t_1\leq s \leq t_2 \leq A}
  |g(s)-g(t_1)|\wedge |g(t_2)-g(s)| \; .
\end{align*}
Since the $Y_i$s are independent conditionally on $\mathcal X$, by
elementary computations similar to those that lead to
\cite[Inequality 13.17]{billingsley:1968}, we obtain that
\begin{align}
   \label{eq:3Qn}
   \esp\left[|R_n'(s)-R_n'(t_1)|^2 |R_n'(t_2)-R_n'(s)|^2 \mid \mathcal X\right]
    \leq 3
   \{Q_n(t_1)-Q_n(t_2)\}^2 \; ,
\end{align}
where
\begin{align*}
  Q_n(s) = \frac{1}{n\bar F(u_n)} \sum_{j=1}^n \bar F_Z(u_n(1+s)/\sigma(X_j)) \; .
\end{align*}
Note that $Q_n(s)$ converges in probability to $T(s)$ which is a
continuous decreasing function on $[0,\infty)$.  Let $m\geq1$ be an
integer and set $\delta=A/2m$. Applying \cite[Theorem
12.5]{billingsley:1968} and using the same arguments as in the proof
of \cite[Theorem 15.6]{billingsley:1968} (p. 129, Eq. (15.26); note
that the assumed continuity of the function $F$ that appears therein
is not used to obtain (15.26)), we see that the bound~(\ref{eq:3Qn})
yields, for some constant $C$ (whose numerical value may change upon each
appearance),
\begin{align*}
  \pr(w_A^{\prime\prime}(R_n',\delta)>\epsilon \mid \mathcal X) & \leq C
  \epsilon^{-4} \sum_{k=0}^{2m-1} \{Q_n(k\delta)-Q_n((k+2)\delta)\}^2  \\
  & \leq C \epsilon^{-4} Q_n(0) \max_{0\leq k \leq 2m-1}
  \{Q_n(k\delta)-Q_n((k+2)\delta)\} \; .
\end{align*}
Letting now $n\to\infty$ yields
\begin{multline*}
  \limsup_{n\to\infty} \pr(w_A^{\prime\prime}(R_n',\delta)>\epsilon \mid
  \mathcal X)  \\
  \leq C \epsilon^{-4} \max_{0\leq k \leq 2m-1} \{T(k\delta) - T((k+2)\delta)\}
  \leq C \epsilon^{-4} \delta^{\alpha\wedge1} \; .
\end{multline*}
By bounded convergence, this yields
\begin{align*}
  \limsup_{n\to\infty} \pr(w_A^{\prime\prime}(R_n',\delta)>\epsilon ) \leq C
  \epsilon^{-4} \delta^{\alpha\wedge1} \; ,
\end{align*}
and~(\ref{eq:bill-criterion-15.4}) follows.

We prove now tightness of $S_n$. Assume first $1-q(1-H)>1/2$ and
define $S_n'=\rho_n^{-q/2}S_n$. Applying
(\ref{eq:variance-inequality-lrd}) there exists a constant $C$,
which depends only on the Gaussian process $\{X_j\}$, such that we
have, for $s\leq t$,
\begin{align*}
  \var(S_n'(s)-S_n'(t)) & \leq C \var(G_n(X_1,s)-G_n(X_1,t)) \\
  \leq & C \esp \left[ \frac{\pr^2(u_n(1+s) \leq \sigma(X_1) Z_1 \leq u_n(1+t) \mid
    \mathcal X)}{\pr(Z>u_n)} \right]
\end{align*}
Let the expectation in last term be denoted by $Q'_n(s,t)$. By the
same adaptation of the proof of \cite[Theorem
15.6]{billingsley:1968} as previously (see also
\cite{genest:ghoudi:remillard:1996} for a more general extension),
we obtain, for each $A>0$, and for $\delta=A/2m$ for an integer
$m\geq1$,
\begin{align*}
  \pr(w_A^{\prime\prime}(S_n',\delta)>\epsilon) \leq C \epsilon^{-2}
  \sum_{k=0}^{2m-1} Q'_n(2k\delta,(2k+2)\delta) \; .
\end{align*}
Thus, letting $n$ tend to infinity while keeping $m$ fixed, we get
\begin{align*}
  \limsup_{n\to\infty} \pr(w_A^{\prime\prime}(S_n',\delta) > \epsilon) & \leq C
  \epsilon^{-2} \sum_{k=0}^{2m-1}\{(1+2k\delta)^{-\alpha} -
  (1+(2k+2)\delta)^{-\alpha} \}^2 \\
& \leq C \epsilon^{-2} \delta^2 \sum_{k=0}^{2m-1}
(1+2k\delta)^{-2\alpha-2} \leq C\epsilon^{-2} \delta \; .
\end{align*}
Thus $\lim_{\delta\to0} \limsup_{n\to\infty}
\pr(w_A^{\prime\prime}(S_n',\delta)
> \epsilon) = 0$ and this concludes the proof of tightness.

\subsection{Proof of Corollary \ref{cor:practical} and Theorem
  \ref{thm:practical-1} }
As in case of Theorem \ref{thm:general}, we start some heuristic.
Recall computation from Section \ref{sec:heuristic} and the form of
the limiting distribution $w-T\cdot w(0)$. Then
\begin{eqnarray*}
\lefteqn{\Var(\hat T_n(s)) =(1+o(1))\Var(\tilde T_n(s)-T(s)\tilde
T_n(0))}\\
&=& (1+o(1))\frac{1}{n\bar F(u_n)}T(s)(1-T(s))+o(1)T^2(s)\rho_n^{q}.
\end{eqnarray*}
This suggests that in LRD zone $\rho_n^{-q/2}\hat T_n(\cdot)$
converges to 0.\\

To prove it formally, denote $\bar T_n = T_n-T$ and
$\xi_{n}=\frac{Y_{n-k:n} - u_n}{u_n} = \tilde T_n^{\leftarrow}(1)$.
Then $\tilde T_n(\xi_n)=1$, and we have
\begin{align*}
  1 = e_n(\xi_n) + T_n(\xi_n) = e_n(\xi_n) + \bar T_n(\xi_n) +
  T(\xi_n) \; .
\end{align*}
Thus,
\begin{align}
  \label{eq:ecriture-xi_n}
  T(\xi_n) - 1 = - e_n(\xi_n) - \bar T_n(\xi_n) \; .
\end{align}
For any $s \geq 0$, $\hat T_n(s) = \tilde T_n(s+\xi_n(1+s))$ and
$T(s+\xi_n(1+s)) = T(s)T(\xi_n)$, thus
\begin{align*}
  \hat e_n^*(s) & = e_n(s+\xi_n(1+s)) + \bar T_n(s+\xi_n(1+s)) +
  T(s+\xi_n(1+s)) - T(s) \\
  & = e_n(s+\xi_n(1+s)) + T(s)\{T(\xi_n)-1\} + \bar T_n(s+\xi_n(1+s))
  \; .
\end{align*}
Plugging~(\ref{eq:ecriture-xi_n}) into this decomposition of $\hat
e_n^*$, we get
\begin{align}
  \hat e_n^*(s)
  & = e_n(s+\xi_n(1+s)) - T(s) e_n(\xi_n) + \bar T_n(s+\xi_n(1+s)) -
  T(s) \bar T_n(\xi_n) \; . \label{eq:expansion-clean}
\end{align}
In order to prove Corollary~\ref{cor:practical}, we write
\begin{align}
  w_n \hat e_n^*(s) & = w_n \{ e_n(s+\xi_n(1+s)) - T(s) e_n(\xi_n) \} +
  O(w_n\|T_n-T\|_\infty) \; . \label{eq:proof-practical}
\end{align}
Since the convergence in Theorem \ref{thm:general} is uniform, and
by Corollary \ref{coro:intermediate} $\xi_n=o_P(1)$, the first term
in~(\ref{eq:proof-practical}) converges in $D([0,\infty))$ to
$w-T\cdot w(0)$. Under the second order
condition~(\ref{eq:second-ordre-unprimitive}), the second term is
$o(1)$.  This concludes the proof of Theorem~\ref{cor:practical}.

We now prove Theorem~\ref{thm:practical-1}. In order to study the
second-order asymptotics of $w_n \hat e_n^*(s)$, we need precise
expansion for $e_n(s+\xi_n(1+s))$ and $e_n(\xi)$. For this we will
use the expansions of the tail empirical process in
Section~\ref{sec:fidi}.
Since $\bar F(u_n)=k/n$, using~(\ref{eq:decomposition-sv}),
(\ref{eq:expansion}) and (\ref{eq:claim-2}), we have
\begin{gather}
  e_n(s) = R_n(s) + \frac{ \bar F_Z(u_n)}{n\bar F(u_n)} \, T(s) S_n^*
  + \frac{ \bar F_Z(u_n)}{n\bar F(u_n)} \tilde S_n(s) \; ,
\label{eq:expansion-4}
\end{gather}
which, noting again that $T(s+\xi_n(1+s)) = T(s)T(\xi_n)$, yields
\begin{align*}
  e_n(s+\xi_n(1+s)) - T(s) e_n(\xi_n) & = R_n(s+\xi_n(1+s)) - T(s)
  R_n(\xi_n)  \\
  & \ \ \ + \frac{ \bar F_Z(u_n)}{n\bar F(u_n)} \{ \tilde
  S_n(s+\xi_n(1+s)) - T(s) \tilde S_n(\xi_n) \} \;
\end{align*}
and
\begin{multline}
  \label{eq:decomp-non-uniform}
  \hat e_n^*(s) = R_n(s+\xi_n(1+s)) - T(s) R_n(\xi_n) \\
  + \frac{ \bar F_Z(u_n)}{n\bar F(u_n)} \{ \tilde S_n(s+\xi_n(1+s)) -
  T(s) \tilde S_n(\xi_n) \} + \bar T_n(s+\xi_n(1+s)) - T(s) \bar
  T_n(\xi_n) \; .
\end{multline}
Similarly to (\ref{eq:control-Hermite}), and utilising $\bar
F_Z(u_n)/\bar F(u_n) = O(1)$,
\begin{align*}
  \var \left( \frac{ \bar F_Z(u_n)}{n\bar F(u_n)} \tilde S_n(s)
  \right) \leq C  \{\rho_n^q \vee \ell_1(n)n^{-1} \}
  \|G_n(\cdot,s) - T(s)G(\cdot)\|_{L^2(\mu)}^2 \; .
\end{align*}
Using the second order Assumption (SO) through~(\ref{eq:rate-G_n}),
we obtain
\begin{align}
  \var \left( \frac{ \bar F_Z(u_n)}{n\bar F(u_n)} \tilde S_n(s)
  \right) = O\left(  \{\rho_n^q \vee \ell_1(n)n^{-1} \}
    \eta^*(u_n)^2\right) = o \left( \eta^*(u_n)^2 \right) \; .
  \label{eq:bias}
\end{align}
Using (\ref{eq:expansion-4}) in the representation
(\ref{eq:expansion-clean}) and since
Proposition~\ref{prop:transfert-second-ordre} implies that
$\|T_n-T\|_\infty = O(\eta^*(u_n))$, we obtain:
\begin{align*}
  \hat e_n^*(s) & = R_n(s+\xi_n(1+s)) - T(s) R_n(\xi_n) +
  O_P(\eta^*(u_n)) \; . %\label{eq:decomp}
\end{align*}
Since we have already proved that the convergence of $\sqrt k R_n$
is uniform, we obtain that $\sqrt k e_n^*$ converges in the sense of
finite dimensional distribution to $B\circ T$, where $B$ is the
Brownian bridge, if the second order
condition~(\ref{eq:negligibility-1a}) holds. To prove tightness, we
only have to prove that $k^{1/2} n^{-1} S_n$ converges uniformly to
zero on compact sets. For $s\geq0$ and $x\in\mathbb R$, denote $\bar
G_n(x,s) = G_n(x,s)-T(s)G(x)$ and recall that we have shown in
Section~\ref{sec:tightness-deterministic-level} that
\begin{align*}
  n^{-2} \var(\tilde S_n(s)-\tilde S_n(s')) & \leq C \| \bar
  G_n(\cdot,s_2) - \bar G_n(\cdot,s_1) \|_{L^2(d\mu)}^2 \; .
\end{align*}
Applying~(\ref{eq:pour-th-random}), we get
\begin{equation}\label{eq:uniform-conv}
  n^{-2} \var(\tilde S_n(s)-\tilde S_n(s'))  \leq C (\eta^*(u_n))^2
  \esp\left [(\sigma(x)\vee1)^{2\alpha(\beta+1)+\epsilon} \right]
  (s-s')^2\; ,
\end{equation}
which proves that $k^{1/2} n^{-1}\tilde S_n$ converges uniformly to
zero on compact sets.

\subsection{Proof of Corollary \ref{cor:Hill}}
Using the decomposition~(\ref{eq:decomp-non-uniform}), and the
identity $\int_0^\infty (1+s)^{-1} \, T(s)\, \d s = \gamma$, we have
\begin{align}
  \hat \gamma_n - \gamma & = \int_0^\infty \frac{\hat e_n^*(s)}{1+s}
  \, \d s = \int_0^\infty \frac{R_n(s+\xi_n(1+s))}{1+s} \, \d s -
  \gamma R_n(\xi_n) \nonumber \\
  & + \frac {\bar F_Z(u_n)}{n\bar F(u_n)} \int_0^\infty \frac{\tilde
    S_n(s+\xi_n(1+s))}{1+s} \, \d s - \gamma \frac {\bar
    F_Z(u_n)}{n\bar F(u_n)} \tilde S_n(\xi_n) \label{eq:toto-1}
  \\
  & + \int_0^\infty \frac{\bar T_n(s+\xi_n(1+s))}{1+s} \, \d s -
  \gamma \bar T_n(\xi_n) \label{eq:toto-2} \; .
\end{align}
We must prove that the terms in~(\ref{eq:toto-1})
and~(\ref{eq:toto-2}) are $O_P(\eta^*(u_n))$ and that
\begin{align}
  \sqrt k \int_0^\infty (1+s)^{-1} R_n(s+\xi_n(1+s)) \, \d s
  \convdistr \int_0^\infty \frac{W \circ T(s)}{1+s} \, \d s = \gamma
  \int_0^1 \frac{W(t)}t \, \d t \; . \label{eq:int-W}
\end{align}
To prove~(\ref{eq:int-W}), we follow the lines of \cite[Section
9.1.2]{resnick:2007}.  We must prove that we can apply continuous
mapping.  To do this, it suffices to establish that for any
$\delta>0$ we have
$$
\lim_{M\to\infty}\limsup_{n\to\infty} A_{n,M} = 0 \; ,
$$
where
$$
A_{n,M} = \pr \left(\sqrt k \int_{M}^{\infty} \left|\frac{1}{k}
    \sum_{j=1}^n \left(1_{\{Y_j>u_ns\}} - P \left(Y_j >
        u_ns|{\cal X}\right)\right)\right| \frac{ds}{s} >
  \delta\right) \; .
$$
By Markov's inequality, conditional independence and Potter's bound
\cite[Theorem 1.5.6]{bingham:goldie:teugels:1989} , we have, for
some $\epsilon>0$,
\begin{align*}
  A_{n,M} \leq C \frac{\sqrt n}{\sqrt k } \int_M^\infty \frac{
    \pr^{1/2}(Y>u_ns)}s \, \d s \leq C \, \sqrt {\frac{n \bar F(u_n)}
    k} \int_M^\infty s^{-1-\alpha/2+\epsilon} \, \d s \leq C
  M^{-\alpha/2+\epsilon} \to 0
\end{align*}
as $M\to\infty$, since $k=n\bar F(u_n)$. This
proves~(\ref{eq:int-W}). To get a bound for~(\ref{eq:toto-2}), we
use~(\ref{eq:rate-non-uniform}) which yields, for all $t\geq0$,
\begin{align*}
  |\bar T_n(t)| \leq C \eta^*(u_n) (1+t)^{-\alpha+\rho\pm\epsilon} \; .
\end{align*}
Thus $\bar T_n(\xi_n) = O_P(\eta^*(u_n))$ and $ |\bar
T_n(s+\xi_n(1+s))| \leq C \eta^*(u_n) (1+s)^{-\alpha+\rho+\epsilon}
(1+\xi_n)^{-\alpha} $, thus
\begin{align*}
  \int_0^\infty \frac{|T_n(s+\xi_n(1+s))|}{1+s} \, \d s =
  O_P(\eta^*(u_n)) \; .
\end{align*}
We finally bound~(\ref{eq:toto-1}).
\begin{align*}
  \int_0^\infty \frac{n^{-1} \tilde S_n(s+\xi_n(1+s))}{1+s} \, \d s &
  = \int_{\xi_n}^\infty \frac{n^{-1} \tilde S_n(u)}{1+u} \, \d u \; .
\end{align*}
Since $\xi_n=o_P(1)$, we can write
\begin{align*}
  \pr\left( k^{1/2} \int_{\xi_n}^\infty \frac{n^{-1} \tilde
      S_n(u)}{1+u} \, \d u > \epsilon\right) & \leq \pr(\xi_n>1) + \pr
  \left( k^{1/2} \int_{1}^\infty \frac{n^{-1} |\tilde S_n(u)|}{1+u}
    \, \d u > \epsilon \right) \\
  & \leq o(1) + \frac{k^{1/2}}{n\epsilon} \int_1^\infty
  \frac{\esp^{1/2}[\tilde S_n^2(s)]}{1+s} \; \d s
\end{align*}
Applying~(\ref{eq:control-Hermite})
and~(\ref{eq:rate-G_n-nonuniform}) yields
\begin{align*}
  \int_1^\infty \frac{n^{-1} \esp^{1/2}[\tilde S_n^2(s)]}{1+s} \, \d s
  & \leq C \rho_n^{q/2} \eta^*(u_n) \int_0^\infty
  s^{-\alpha(\beta+1)/2+\epsilon-1} \, \d s = o_P(k^{-1/2}) \; .
\end{align*}
Thus the first term in~(\ref{eq:toto-1}) is $o_P(k^{-1/2})$, and so
is the second term since $k^{1/2} n^{-1} \tilde S_n$ converges
uniformly to zero on compact sets. This concludes the proof of
Corollary~\ref{cor:Hill}.

\subsection{Second order  regular variation}
\label{sec:proof-s-o-c}

The main tool in the study of the tail of the product $YZ$ is the
following bound. For any $\epsilon>0$, there exists a constant $C$
such that, for all $y>0$,
\begin{equation}\label{eq:bound-1}
  \frac{\pr(yZ_1>x)}{\pr(Z_1>x)} \leq C(1\vee y^{\alpha+\epsilon}) \; .
\end{equation}
This bound is trivial if $y<1$ and follows from Potter's bounds if
$y>1$.

\begin{proof}[Proof of Lemma~\ref{lem:convergence-uniforme-hermite}]
  By Breiman's Lemma, we know that for any sequence $u_n$ such that
  $u_n\to \infty$,
\begin{gather}\label{eq:bound-2}
  \lim_{n\to\infty} G_n(x,s) = \lim_{n\to\infty}
  \frac{\pr(\sigma(x)Z_1>(1+s)u_n)}{\pr(Z>u_n) } ) = \sigma^\alpha(x)
  (1+s)^{-\alpha} = \sigma^\alpha(x) T(s) \; .
\end{gather}
If $\esp[\sigma^{\alpha+\epsilon}(X)]<\infty$, then the
bound~(\ref{eq:bound-1}) implies that the
convergence~(\ref{eq:bound-2}) holds in $L^p(\mu)$ for any $p$ such
that $p\alpha<\alpha+\epsilon$, uniformly with respect to $s$, i.e.
\begin{align*}
  \lim_{n\to\infty} \esp[\sup_{s\geq0} |G_n(X,s) - \sigma^\alpha(X)
  T(s)|^p] = 0 \; .
\end{align*}
\end{proof}
Before proving Proposition~\ref{prop:transfert-second-ordre}, we
need the following lemma which gives a non uniform rate of
convergence.

\begin{lem}
  \label{lem:rate-excess-eta}
  If~(\ref{eq:Pareto-assumption}),~(\ref{eq:representation})
  and~(\ref{eq:borne-eta}) hold, if $\eta^*$ is regularly varying at infinity
  with index $\rho$, for some $\rho\leq0$,
  then~%(\ref{eq:second-order-tightness}) holds and
  for any $\epsilon>0$, there exists a constant $C$ such that
   \begin{gather}
     \forall t\geq1 \;, \ \ \forall z>0 \;, \ \ \left|
       \frac{\pr(Z>zt)}{\pr(Z>t)} -z^{-\alpha} \right| \leq C
     \eta^*(t) z^{-\alpha+\rho} (z\vee z^{-1})^\epsilon \; .
     \label{eq:rate-non-uniform}
   \end{gather}
\end{lem}
\begin{proof}
  Since $\eta^*$ is decreasing, using the bound $|\mathrm e^u-1|\leq u
  \mathrm e^{u_+}$ with $u_+ = \max(u,0)$, we have, for all $z>0$,
  \begin{align}
    \left| \frac{\pr(Z>zt)}{\pr(Z>t)} - z^{-\alpha} \right| & =
    z^{-\alpha} \left| \exp\int_1^z \frac{\eta(ts)}s \, \d s -
      1 \right|  \nonumber \\
    & \leq Cz^{-\alpha} \int_{z\wedge1}^{z\vee1} \frac{\eta^*(st)}s \;
    \d s \; \exp \int_{z\wedge1}^{z\vee1} \frac{\eta^*(st)}s \; \d s
    \nonumber    \\
    & \leq Cz^{-\alpha} \log(z) \, \eta^*(t(z\wedge1)) \; \exp
    \int_{z\wedge1}^{z\vee1} \frac{\eta^*(st)}s \;
    \d s \nonumber    \\
    & \leq C z^{-\alpha} (z\wedge1)^{\rho-\epsilon/2} \, \eta^*(t) \;
    \exp \int_{z\wedge1}^{z\vee1} \frac{\eta^*(st)}s \; \d s \; .
  \label{eq:reste-un-bout}
  \end{align}
  We now distinguish three cases. Recall that $\eta^*$ is decreasing.
\begin{itemize}
\item If $z\geq1$, then $z \to \exp\int_1^z s^{-1} {\eta^*(s)} \, \d
  s$ is a slowly varying function by Karamata's representation
  Theorem, and is $O(z^{\epsilon/2})$ for any $\epsilon>0$. Plugging
  this bound into~(\ref{eq:reste-un-bout})
  yields~(\ref{eq:rate-non-uniform}).
\item If $z<1$ and $tz\geq1$, then
  \begin{align*}
    \exp \int_{z}^{1} \frac{\eta^*(st)}s \; \d s = \exp
    \int_1^{1/z} \frac{\eta^*(stz)}s \, \d s \leq \exp \int_1^{1/z}
    \frac{\eta^*(s)}s \, \d s = O(z^{-\epsilon/2})
  \end{align*}
  for any $\epsilon>0$ by the same argument as above and this
  yields~(\ref{eq:rate-non-uniform}).
\item If $tz<1$, then $t^r \leq z^{-r}$ for any $r>0$ and
  $t^{\rho-\epsilon}=O(\eta^*(t))$ for any $\epsilon>0$.  Thus
\begin{align*}
  \left| \frac{\pr(Z>zt)}{\pr(Z>t)} - z^{-\alpha} \right| & \leq
  \frac1{\pr(Z>t)} + z^{-\alpha} \leq C t^{\alpha+\epsilon/2} +
  z^{-\alpha} \leq C z^{-\alpha-\epsilon/2} \\
  & \leq C z^{-\alpha+\rho-\epsilon} t^{\rho-\epsilon/2} \leq C
  z^{-\alpha+\rho-\epsilon} \eta^*(t) \; .
\end{align*}
\end{itemize}
This concludes the proof of~(\ref{eq:rate-non-uniform}).
\end{proof}

The following bound is used in the proof of prove
Theorem~\ref{thm:practical-1}.
\begin{lem}
  If~(\ref{eq:Pareto-assumption}),~(\ref{eq:representation})
  and~(\ref{eq:borne-eta}) hold, if $\eta^*$ is regularly varying at infinity
  with index $\rho$, for some $\rho\leq0$, then there exists a constant $C$
  such that for all $t\geq1$ and $b>a>0$,
\begin{align}
  \left| \frac{\pr(at < Z \leq bt)}{\pr(Z>t)} -
    (a^{-\alpha}-b^{-\alpha})\right| & \leq C \eta^*(t)
  (a\wedge 1)^{-\alpha+\rho-\epsilon} (b-a) \; .
   \label{eq:pour-th-random}
\end{align}
\end{lem}
\begin{proof}
  The bound~(\ref{eq:pour-th-random}) follows from the following one
  and~(\ref{eq:bound-1}) applied to the function $\eta^*$.
\begin{align}
  \left| \frac{\pr(at < Z \leq bt)}{\pr(Z>t)} -
    (a^{-\alpha}-b^{-\alpha})\right| & \leq C \eta^*((a\wedge 1)t) (a\wedge
  1)^{-\alpha-1-\epsilon} (b-a) \label{eq:tightness-Z-bis}
\end{align}
Let $\ell$ be the function slowly varying at infinity that appears
in~(\ref{eq:Pareto-assumption}), defined on $[0,\infty)$ by $\ell(t)
= t^\alpha\pr(Z>t)$.  Assumption (SO) implies that
\begin{align} \label{eq:representation-ell}
  \ell(t) = \ell(1) \exp\int_1^t \eta(s) \; \frac{\d s}s
\end{align}
 where the function $\eta$ is
measurable and bounded. This implies that the function $\ell$ is the
solution of the equation
\begin{align} \label{eq:equation}
  \ell(t) = \ell(1) + \int_1^t \eta(s) \ell(s) \; \frac{\d s} s \; .
\end{align}
Conversely, if $\ell$ satisfies~(\ref{eq:equation})
then~(\ref{eq:representation-ell}) holds. We first prove the
following useful bound. For any $\epsilon>0$, there exists a
constant $C$ such that for any $t\geq1$ and all $a>0$,
\begin{align} \label{eq:ell(at)/ell(t)}
  \frac{\ell(at)}{\ell(t)} \leq C a^{\pm\epsilon} \; ,
\end{align}
where we denote $a^{\pm\epsilon} = \max(a^\epsilon,a^{-\epsilon})$.
Indeed, if $at\geq1$, then, $\eta^*$ being decreasing, we have
\begin{align*}
  \frac{\ell(at)}{\ell(t)} \leq C\exp \int_{a\wedge1}^{a\vee1}
  \frac{\eta^*(ts)}s \, \d s \leq C\exp \int_1^{a\vee(1/a)}
  \frac{\eta^*(ts)}s \, \d s \leq C a^{\pm\epsilon} \; ,
\end{align*}
since the latter function is slowly varying by Karamata's
representation theorem. If $at<1$, then $\ell(at)\leq 1$ and
$\ell^{-1}(t) = o(t^\epsilon) = o(a^{-\epsilon})$. This
proves~(\ref{eq:ell(at)/ell(t)}).  Next,
applying~(\ref{eq:equation}) and~(\ref{eq:ell(at)/ell(t)}), for any
$\epsilon>0$ and $0<a<b$, we have
\begin{eqnarray}
  \label{eq:l/l}
  \left| \frac{\ell(bt)}{\ell(at)} - 1 \right| & =& \left| \int_a^b
    \eta(st) \, \frac{\ell(st)}{\ell(at)} \, \frac{\d s}s \right|
\leq  Ca^{\pm\varepsilon}\left| \int_a^b
    \eta(st) \, \frac{\ell(st)}{\ell(t)} \, \frac{\d s}s
    \right|\nonumber\\
    &\leq&
  C \eta^*(at) \int_a^b s^{\pm 2\epsilon-1} \; \d s \leq C \eta^*(at) \,
  a^{\pm\epsilon-1} (b-a) \; .
\end{eqnarray}
Applying~(\ref{eq:ell(at)/ell(t)}) and~(\ref{eq:l/l}), we also
obtain
\begin{align}
  \label{eq:l/l-a}
  \left| \frac{\ell(at)}{\ell(t)} - 1 \right| & \leq C
  \eta^*((a\wedge1)t) \, a^{\pm\epsilon} \; .
\end{align}
For $\epsilon>0$ and $0<a<b$, we have
\begin{align*}
  \frac{\pr(at < Z \leq bt)}{\pr(Z>t)} - (a^{-\alpha} - b^{-\alpha}) &
  = a^{-\alpha} \left\{ \frac{\ell(at)}{\ell(t)} - 1 \right\} -
  b^{-\alpha} \left\{\frac{\ell(bt)}{\ell(t)}-1\right\} \ \\
  & = (a^{-\alpha} - b^{-\alpha}) \left\{ \frac{\ell(at)}{\ell(t)} - 1
  \right\} - b^{-\alpha} \frac{\ell(at)}{\ell(t)}
  \left\{\frac{\ell(bt)}{\ell(at)}-1\right\} \; ,
\end{align*}
which yields
\begin{align*}
  \left| \frac{\pr(at < Z \leq bt)}{\pr(Z>t)} - (a^{-\alpha} -
    b^{-\alpha}) \right| & \leq C \eta^*((a\wedge1)t)
  a^{\alpha-1\pm\epsilon} (b-a) \; .
\end{align*}

\end{proof}

\begin{proof}[Proof of Proposition~\ref{prop:transfert-second-ordre}]
  Define the function $\bar \sigma$ by $\bar \sigma(x) =
  \sigma(x)\vee1$. Applying~(\ref{eq:rate-non-uniform}) with
  $(1+s)/\sigma(x)$ instead of $z$ and $u_n$ for $t$, we get
\begin{align}
  \left|G_n(x,s) - \sigma^\alpha(x)T(s) \right| & = \left|
    \frac{\pr(\sigma(x)Z>u_n(1+s))}{\pr(Z>u_n)} - \sigma^\alpha(x)
    T(s) \right| \nonumber \\
  & \leq C \eta^*(u_n) \bar\sigma(x)^{\alpha(\beta+1)+\epsilon}
  (1+s)^{-\alpha(\beta+1)+\epsilon} \; . \label{eq:rate-G_n-nonuniform}
\end{align}
This implies, for all $p$ such that
$\esp[\sigma^{p\alpha(\beta+1)+\epsilon}(X)]<\infty$, that
\begin{align*}
  \esp \left[ \sup_{s\geq1} |G_n(X,s) - T(s) \sigma^\alpha(X)|^p
  \right] = O(\{\eta^*(u_n)\}^p) \; .
\end{align*}
This proves~(\ref{eq:rate-G_n}) which in turn
implies~(\ref{eq:rate-T_n}) since $T_n(s) = \frac{\bar F(u_n)}{\bar
F_Z(u_n)} \esp[G_n(X,s)]$.  In order to prove that $\bar F_Y \in
2RV(-\alpha,\eta^*)$, denote $\tilde \ell(y) = y^\alpha\pr(Y>y)$. We
will prove that there exists a measurable function $\tilde\eta$ such
that~(\ref{eq:equation}) holds with $\tilde\ell$ and $\tilde \eta$.
Denote $\xi=\sigma(X)$.  Applying~(\ref{eq:equation}) and using the
independence of $\xi$ and $Z$, we have
\begin{align*}
  \tilde \ell(y) & = \esp[\xi^\alpha \ell(y/\sigma)] = \ell(1)
  \esp[\xi^\alpha] + \esp\left[\xi^\alpha \int_1^{y/\xi} \eta(s)
    \ell(s) \frac{\d s}s \right] \\
  & = \ell(1) \esp[\xi^\alpha] + \esp\left[\xi^\alpha \int_\xi^{y}
    \eta(s/\xi) \ell(s/\xi) \frac{\d s}s \right] \\
  & =\esp\left[ \xi^\alpha \left\{ \ell(1) -\int_1^\xi \eta(s/\xi)
      \ell(s/\xi) \frac{\d s}s \right\} \right] + \esp\left[\xi^\alpha
    \int_1^{y} \eta(s/\xi) \ell(s/\xi) \frac{\d s}s \right] \\
  & =\esp\left[ \xi^\alpha \left\{ \ell(1) + \int_{1/\xi}^1 \eta(s)
      \ell(s) \frac{\d s}s \right\} \right] + \int_1^{y}
  \esp[\xi^\alpha \eta(s/\xi) \ell(s/\xi)] \frac{\d s}s \\
  & =\esp \left[ \xi^\alpha \ell(1/\xi) \right] + \int_1^{y}
  \esp[\xi^\alpha \eta(s/\xi) \ell(s/\xi)] \frac{\d s}s = \tilde
  \ell(1) + \int_1^t \tilde \eta(s) \tilde \ell(s) \frac{\d s}s \; ,
\end{align*}
where we have defined
\begin{align*}
  \tilde \eta(s) = \frac{\esp[\xi^\alpha \eta(s/\xi)
    \ell(s/\xi)]} {\esp[\xi^\alpha \ell(s/\xi)]} =
  \frac{\esp[\xi^\alpha \eta(s/\xi) \ell(s/\xi)/\ell(s)]}
  {\esp[\xi^\alpha \ell(s/\xi)/\ell(s)]} \; .
\end{align*}
The denominator of the last expression is bounded away from zero.
Indeed, let $\epsilon>0$ be such that $\pr(\xi\geq\epsilon)>0$. Then
\begin{align*}
  \esp[\xi^\alpha \ell(s/\xi)/\ell(s)] = \frac{\pr(\xi
    Z>s)}{\pr(Z>s)} \geq \frac{\pr(\xi\geq\epsilon) \pr(Z>s/\epsilon)}
  {\pr(Z>s)} \; .
\end{align*}
Since $Z$ has a regularly varying tail, it holds that $\inf_{s\geq0}
\pr(Z>s/\epsilon)/\pr(Z>s) >0$. This proves our claim. Thus,
applying~(\ref{eq:bound-1}) with the regularly varying function
$\eta^*$, we get, for $\epsilon>0$ such that
$\exp[\xi^{\alpha-\rho+\epsilon}]<\infty$,
\begin{align*}
  |\tilde \eta(s)| \leq C \eta^*(s) \, \esp[\xi^\alpha
  \{\eta^*(s/\xi)/\eta^*(s)\} \{\ell(s/\xi)/\ell(s)\}] \leq C
  \eta^*(x) \, \esp[\xi^\alpha(\xi\vee1)^{-\rho+\epsilon}]\; .
\end{align*}
Thus $\tilde \ell$ satisfies equation~(\ref{eq:equation}) with
$\tilde\eta$ such that $|\eta|\leq C\eta^*$, thus $Y\in
2RV(-\alpha,\eta^*)$.
%Since Condition (SO)
%implies~(\ref{eq:second-order-tightness}), this concludes the proof of
%Proposition~\ref{prop:transfert-second-ordre}.
\end{proof}

\section*{Acknowledgement}
The research of the first author was supported by NSERC grant. The
research of the second author is partially supported by the ANR
grant ANR-08-BLAN-0314-02.

%  \bibliography{bibSV}{}
%  \bibliographystyle{plain}

% % \bibliographystyle{alpha}

\end{document}